%% file: wc_vs_bc.tex
\documentclass[a4paper]{amsart}
\input{GenCommands}

\usepackage{amscd}

\MakeComments{\tnote}{Thomas}{Purple}
\MakeComments{\fnote}{Fede}{ForestGreen}
\MakeComments{\cnote}{Christos}{Orange}

\newcommand{\naturals}{\mathbb{N}}

\newcommand{\expnat}{{2^{\mathbb N}}}
\newcommand{\expeven}{{2^{\text{ev}}}}
\newcommand{\expodd}{{2^{\text{od}}}}
\newcommand{\iotaeven}{{\iota_{\text{ev}}}}
\newcommand{\iotaodd}{{\iota_{\text{od}}}}

\newcommand{\jalt}{{j}_{\rm alt}}
\newcommand{\shift}{\mathcal S}
\newcommand{\pt}{{\rm pt}}

\newcommand{\into}{\hookrightarrow}

\newcommand{\Croe}[1]{C^*_{\rm Roe}(#1)}

\newcommand{\Cfp}[1]{C^*_{\rm fp}(#1)}
\newcommand{\Clc}[1]{C^*_{\rm lc}(#1)}
\newcommand{\wc}{\mathcal{WC}}
\newcommand{\cone}{\mathcal{O}}
\newcommand{\wcone}{\mathcal{O}_\Gamma}
\newcommand{\iwcone}{\wcone^{\mathbb N+ 1}}
\newcommand{\rpos}{{\mathbb R}_{\geq 1}}
\newcommand{\npos}{{\mathbb N}_{\geq 1}}

\DeclareMathOperator{\topp}{top}
\DeclareMathOperator{\bott}{bot}
\DeclareMathOperator{\up}{up}
\DeclareMathOperator{\down}{down}

\begin{document}

\title[Coarse Baum--Connes and warped cones]{Coarse Baum--Connes and warped cones: failure of surjectivity in odd degree}
\date{\today}

\begin{abstract}
  We prove a conjecture of Roe by constructing unified warped cones that violate the coarse Baum–Connes conjecture. Interestingly, the reason for this is probably not what Roe expected, as the obstruction arises in odd rather than even degree.
\end{abstract}

\author{Christos Kitsios}
\address{Mathematisches Institut, Georg-August-Universit\"{a}t G\"{o}ttingen, Bunsenstr. 3-5, 37073 G\"{o}ttingen, Germany.}
\email{christos.kitsios@uni-goettingen.de}

\author{Thomas Schick}
\address{Mathematisches Institut, Georg-August-Universit\"{a}t G\"{o}ttingen, Bunsenstr. 3-5, 37073 G\"{o}ttingen, Germany.}
\email{thomas.schick@uni-goettingen.de}

\author{Federico Vigolo}
\address{Mathematisches Institut, Georg-August-Universit\"{a}t G\"{o}ttingen, Bunsenstr. 3-5, 37073 G\"{o}ttingen, Germany.}
\email{federico.vigolo@uni-goettingen.de}

\maketitle


The study of \emph{large scale features} of metric spaces is an important
aspect of geometry. Often, only the large scale geometry is canonically
defined, as in the case of finitely generated groups and their Cayley
graphs. In other situations, only the large scale features are relevant, or it
is beneficial to concentrate on them to single out the most important aspects,
as happens often in large scale index theory.

We concentrate on proper metric spaces, i.e., spaces where closed bounded subsets are
compact. 
With any proper metric space $X$ is associated a $C^*$-algebra of special interest, namely its \emph{Roe algebra} $\Croe{X}$.
The original motivation for introducing such algebras stemmed from index-theoretic considerations \cites{Roe1996IndexTheory,Roe1993CoarseCohomology}, and their modern definition was formalized in \cites{HigsonPedersenRoe1997AlgebrasControlled,HigsonRoeYu1993CoarseMayerVietoris}.
As it turns out, these $C^*$-algebras can be interpreted as an analytic
counterpart to the large scale geometry of metric spaces. Namely, it is a
deep result that two proper metric spaces are coarsely equivalent if and only if
their Roe algebras are isomorphic
\cites{MartinezVigolo2024RigidityFramework,MartinezVigolo2024Crigidity}.

One very interesting large scale invariant of a proper metric spaces $X$ is the
K\=/theory of its Roe algebra, which is the receptacle of large scale index
invariants. These indices, and hence $K_*(\Croe X)$, contain important geometric
information about $X$. However, the K\=/theory of $C^*$-algebras is an intricate object
and often not easy to get ones' hands on. 
A more topological counterpart to it is the
\emph{coarse K-homology} $KX_*(X)$ of $X$  \cite{Roe1993CoarseCohomology}.
There is a natural homomorphism
\[
    \mu_c \colon KX_*(X)\to K_*(\Croe X)
\]
connecting the two sides. This was constructed by Higson--Roe and Yu, and is known as the \emph{coarse assembly map}
\cites{Roe1993CoarseCohomology,HigsonRoe1995CoarseBaum,Yu1995CoarseBaum}. The
coarse Baum--Connes conjecture, which is due to Higson--Roe and Yu, posits that $\mu_c$ is an isomorphism whenever
$X$ is a metric space of bounded geometry. The name ``coarse Baum--Connes'' is justified by analogy with the famous Baum--Connes Conjecture on universal spaces for proper actions \cites{BaumConnes2000GeometricKtheory,BaumConnesHigson1994ClassifyingSpace}.

Notably, if $\mu_c$ is an isomorphism for a given metric space $X$, then this space must also satisfy a whole range of deep results. It can indeed be argued that the existence of the coarse assembly map is a prime source of motivation for being interested in Roe algebras in the first place. 
We refer to \cites{WillettYu2020HigherIndex} for a comprehensive introduction to the subject.

\smallskip

It has been known for a long time that the coarse Baum--Connes conjecture is
alas not true in general \cites{HigsonLafforgueSkandalis2002CounterexamplesBaum}.
However, there are very few examples of spaces for which $\mu_c$ is not an isomorphism. Much the contrary is true: most `reasonable' spaces do satisfy the coarse Baum--Connes conjecture \cite{Yu2000CoarseBaum}.
As a matter of fact, the only 
obstruction to the coarse Baum--Connes conjecture known so far is the existence of \emph{non-compact ghost projections} in $\Croe{X}$.\footnote{Note that we only stated the coarse Baum--Connes conjecture for spaces of bounded geometry. For spaces of unbounded geometry a few other obstructions are known.}
These projections are inherently `global' objects and the K\=/theory classes they determine are generally expected to lie outside the image of $\mu_c$.

The main example of spaces having non-compact ghost projections are built taking coarse disjoint unions of expander graphs, and in some cases it was indeed shown that these spaces violate the coarse Baum--Connes conjecture. This is how the first counterexamples were found \cites{HigsonLafforgueSkandalis2002CounterexamplesBaum}.
Subsequent works have enlarged the class of expanders known to violate the surjectivity of $\mu_c$ \cites{WillettYu2012HigherIndex,Finn-Sell2014FibredCoarse} and identified other classes of metric spaces for which surjectivity fails \cites{LiVigoloZhang2023AsymptoticExpansion,LiVigoloZhang2023MarkovianRoealgebraic,KhukhroLiVigolo2021StructureAsymptotic,LiSpakulaZhang2023MeasuredAsymptotic,Sawicki2018GeometryMetric}.
However, these developments all rely on similar principles: constructing spaces as coarse disjoint unions admitting non-compact ghost projections analogous to the case of expanders.
This is far from being a satisfactory picture.

\smallskip

In this context, one very interesting space to consider is Roe's \emph{unified warped cone} \cites{Roe1995FoliationsCoarse,Roe1996IndexTheory,Roe2005WarpedCones}.
The construction is as follows. Given a (compact) Riemannian manifold $M$ with metric tensor $\varrho$, its \emph{cone} $\cone M$ is the manifold $M\times \rpos$ equipped with the Riemannian metric $t^2\varrho + dt^2$, where $t$ is the $\rpos$-coordinate.
Let now $\Gamma$ be a group equipped with a fixed finite symmetric generating set $S\subset \Gamma$, and let $\Gamma\curvearrowright M$ be an action by homeomorphisms. The \emph{(unified) warped cone} $\wcone M$ is the metric space $(M\times\rpos, \delta^\Gamma)$, where
$\delta^\Gamma$ is defined as the largest metric satisfying:
\[
\delta^\Gamma\leq d_{\cone M}
\quad\text{and}\quad
\delta^\Gamma((x,t),(s\cdot x,t))\leq 1
\]
where $d_{\cone M}$ is the metric on $\cone M$ and $x,x'\in M$, $s\in S$ are arbitrary.
That is, $\delta^\Gamma$ is obtained by warping the metric $d_{\cone M}$ by adding shortcuts along the group action (alternatively, $\delta^\Gamma$ is defined by imposing that all the orbit maps $\Gamma\to \wcone M$ be $1$-Lipschitz with respect to the word metric). It is easy to verify that $\delta^\Gamma$ is well\=/defined.

The warped cone $\wcone M$ has a rather nice local geometry. For instance, if the action is free and by isometries, one can show that the ball of radius $R$ centered at a point $(x,t)$ will converge (in the Gromov--Hausdorff sense) to the ball of radius $R$ in $\mathbb R^{\dim(M)}\times \Gamma$ as $t\to\infty$
\cites{deLaatVigolo2019SuperexpandersGroup,Vigolo2019DiscreteFundamental}.
On the other hand, the global geometric features of $\wcone M$ are very sensitive to the dynamical properties of the action \cites{Roe2005WarpedCones,FisherNguyenvanLimbeek2019RigidityWarped,NowakSawicki2017WarpedCones,Vigolo2019DiscreteFundamental}.
In particular, Roe claimed already in \cite{Roe2005WarpedCones} that warped cones could be used to construct counterexamples to the coarse Baum--Connes conjecture---this was one of Roe's primary motivations for introducing this construction.
This conjecture was later made precise by Dru\k{t}u and Nowak as follows:

\begin{alphconj}[Roe, Dru\k{t}u--Nowak \cite{DrutuNowak2019KazhdanProjections}*{Conjecture 7.7}]\label{conj:roe}
Let $M$ be a compact Riemannian manifold and $\Gamma\curvearrowright M$ a measure-preserving action by diffeomorphisms. If the action has a spectral gap, then $\mu_c$ is not surjective for $\wcone M$.
\end{alphconj}

If this conjecture was true, this would be a very interesting source of examples indeed. For instance, if $\Gamma =\pi_1(N)$ for some compact manifold $N$, then $\wcone M$ would be coarsely equivalent to a complete Riemannian manifold with compact boundary (to see this, it is enough to observe that $\wcone M$ is coarsely equivalent to a ``foliated warped cone'' \cite{Roe2005WarpedCones}*{Lemma 1.12}).
Replacing the boundary with a cusp would also yield a complete Riemannian
manifold without boundary and with compact core, which by a Mayer--Vietoris argument would still violate the coarse Baum--Connes conjecture.
This setup is quite different from the previous source of counterexamples.

Examples of actions as above are easily obtained by appropriately taking two cocompact lattices $\Gamma,\Lambda$ in some higher rank Lie group $G$ and let $\Gamma$ act on $M\coloneqq G/\Lambda$.
An even simpler example is to take $\Gamma$ to be a non-abelian free subgroup
of $\SU(2,\overline{\mathbb Q} )$ and let $\Gamma\curvearrowright\SU(2,\mathbb
C)$ be the action by left multiplication. This free, isometric action has
spectral gap by deep work of Bourgain--Gamburd
\cite{BourgainGamburd2007SpectralGap}.  

\smallskip

The first result of this paper appears to provide evidence against \cref{conj:roe}. Namely, the motivation behind the claim of Dru\k{t}u--Nowak (and presumably also Roe's) is that under those assumptions the Roe algebra $\Croe{\wcone M}$ can be shown to contain non-compact ghost projections. We will however show that these projections vanish in K\=/theory:

\begin{alphthm}[\cref{thm:vanishing}]\label{thm:vanishing}
    If $\Gamma\curvearrowright M$ is ergodic and $\mathfrak G\in \Croe{\wcone
      M}$ is a generalized Dru\k{t}u--Nowak projection, then $[\mathfrak G]=0$
    in $K_0(\Croe{\wcone M})$.
\end{alphthm}

We prove \cref{thm:vanishing} combining the naturality properties of K\=/theory with a change of perspective coming from \cites{deLaatVigoloWinkel2023DynamicalPropagation,MartinezVigolo2023RoelikeAlgebras}.

\smallskip

As it turns out, the situation is however not at all as dire as it may appear. 
In fact, the vanishing in K\=/theory of the Dru\k{t}u--Nowak projection can be explained via a Mayer--Vietoris argument, which also shows the path to follow to save \cref{conj:roe}.
In turn, we can then confirm the conjecture of Roe by proving the following.

\begin{alphthm}\label{thm:cBC_counter}
  Let $\Gamma$ be a group with property A and $\Gamma\curvearrowright M$ a free and strongly ergodic action by diffeomorphisms. Then the coarse Baum--Connes map
 \begin{equation*}
      \mu_c\colon KX_1(\wcone M)\to K_1(\Croe{\wcone M})
  \end{equation*}
  is not surjective.
\end{alphthm}

\begin{remark*}
  \cref{thm:cBC_counter} does not entirely align with \cref{conj:roe}: on the one hand we requires a few extra conditions (property A, freeness), on the other hand we do not require the full power of spectral gap (strongly ergodic actions suffice).
  Note that the action $\Gamma\curvearrowright \SU(2,\mathbb C)$ described above satisfies the hypotheses of \cref{thm:cBC_counter}.
\end{remark*}

\begin{remark*}
  The failure of the coarse Baum--Connes map being an isomorphism
  we presently uncovered is indeed of a rather different flavor from what previously known,
  proving that Roe's hope here is actually satisfied.
  To our knowledge, this is the first time that the failure is shown to be in
  odd degree in a natural way (i.e.~without applying a cheap trick like a
  suspension to simply shifts degrees).

  Informally speaking, the K$_1$-classes of \cref{thm:cBC_counter} which
  do not lie in the image of the coarse Baum--Connes map are represented by
  ``ghost unitaries''. These might be related to unitaries conjugating the
  ``averaging'' Dru\k{t}u--Nowak ghost 
  projection to the trivial projection in $\Croe{\wcone M}$ to
  realize its vanishing in $K_0(\Croe{\wcone M})$. We leave it for further
  investigation to potentially make this statement precise.
\end{remark*}

To explain the strategy of proof, we introduce the following.

\begin{notation*}
  For any $A\subseteq\rpos$, we let
  \[
    \wcone^AM\coloneqq M\times A\subseteq \wcone M
  \]
  equipped with the restriction of the metric of $\wcone M$.
\end{notation*}

With this at hand, the main computation we need is the following result.

\begin{alphthm}\label{thm:LES}
  Let $HX_{\ast}$ be a coarse homology theory such that the embedding
  $\{1\}\hookrightarrow\expnat$ induces an injection $HX_\ast(\pt)\hookrightarrow
  HX_\ast(\expnat)$. Then coarse Mayer--Vietoris yields a long exact sequence
  \[
  \begin{tikzcd}
    &[0em]
    \hspace{8em}\cdots\ar{r}
    \ar[draw=none]{d}[name=X, anchor=center]{}
    &[-1em]
    HX_{\ast+1}(\wcone M)
    \ar[rounded corners,
      to path={ -- ([xshift=2ex]\tikztostart.east)
        |- (X.center) \tikztonodes
        -| ([xshift=-2ex]\tikztotarget.west)
        -- (\tikztotarget)}]{dll}[at end,swap]{\partial}
    &[-1.5em]
    \\
    HX_\ast(\wcone^\expnat M)\ar{r}{\id+\shift_\ast}
    &
    HX_\ast(\wcone^\expnat M)\ar{r}{\jalt}
    &
    HX_\ast(\wcone M)\ar{r}
    & \cdots
  \end{tikzcd}
  \]
  where $\shift\colon\wcone^\expnat M\to\wcone^\expnat M$ denotes the right shift and $\jalt$ is induced by the inclusion with sign alternating with the parity of the exponent.
\end{alphthm}

We defer to \cref{sec: Mayer--Vietoris} for a details regarding coarse homology theories and Mayer--Vietoris. For now, the important point is that \cref{thm:LES} applies to both the coarse K-homology and the K\=/theory of the Roe algebra. What is more, the coarse assembly map commutes with the resulting long exact sequences. In turn, this yields:

\begin{alphcor}\label{cor:mainplus}
  There is a commutative diagram:

  \begin{equation*}
    \begin{tikzcd}
      KX_{*+1}(\wcone M) \ar[->>]{r}{\partial}\ar{d}{\mu_c}
      &
      \ker\left(
      \id+\shift_*\colon {KX}_*(\wcone^{2^\naturals}M)\to{KX}_*(\wcone^{2^\naturals}M)
      \right)
      \ar{d}{\mu_c}
      \\
      K_{*+1}(\Croe{\wcone M}) \ar[->>]{r}{\partial}
      &
      \ker\left(
      \id+\shift_*\colon  K_*(\Croe{\wcone^{2^\naturals}M})\to K_*(\Croe{\wcone^{2^\naturals}M})
      \right)
    \end{tikzcd}
  \end{equation*}
\end{alphcor}

At this point we can use the techniques of \cites{WillettYu2012HigherIndex,Sawicki2017WarpedCones,LiVigoloZhang2023MarkovianRoealgebraic} to leverage the existence of Dru\k{t}u--Nowak projections in the Roe algebra of warped cones of the form $\wcone^{\expnat}M$ to assemble a K\=/theory class that belongs to the kernel of 
\[
  \id+\shift_\ast \colon K_0(\Croe{\wcone^{2^\naturals}M})\to K_0(\Croe{\wcone^{2^\naturals}M})
\]
but is not in $\mu_c(KX_0(\wcone^\expnat))$.
The proof of \cref{thm:cBC_counter} is then completed picking a preimage under $\partial$ to obtain a class in $K_1(\Croe{\wcone M})\smallsetminus\mu_c(KX_1(\wcone M))$.

\medskip
\noindent\textbf{Structure of the paper.}
In \cref{sec: Background} we cover some preliminaries on Roe algebras and illustrate some known partial results towards \cref{conj:roe}. The proof of \cref{thm:vanishing} is in \cref{sec: vanishing proof}.

In \cref{sec: Mayer--Vietoris} we quickly provide some background on coarse
homology theories, and then carry out the key Mayer--Vietoris computations
that prove \cref{thm:LES}. In \cref{sec: consequences of LES} we prove \cref{thm:cBC_counter} and we also show how \cref{thm:LES} can be used to provide an alternative proof for the vanishing of the Dru\k{t}u--Nowak projection in K-theory (\cref{cor: vanishing II}).

To contain the length of this note, we will not discuss the coarse geometric preliminaries, nor will we provide detailed proofs concerning geometric properties of warped cones. Besides Roe's paper \cite{Roe2005WarpedCones}, useful references for the latter are \cites{Vigolo2018GeometryActions,Sawicki2018GeometryMetric}.

\begin{remark*}
    \begin{enumerate}
        \item The warped cone $\wcone X$ can be defined for an arbitrary (compact) metric space $X$. \cref{conj:roe} is also stated in this more general setting, but in this paper we prefer to restrict to actions on compact Riemannian manifolds to limit the amount of technicalities. The arguments we outline extend without difficulty to the general setup.
        \item Sometimes the term ``warped cone'' is used to denote the family of its level-sets $\wc(\Gamma\curvearrowright M)\coloneqq(M\times\{t\}\mid t\in \rpos)$. To make the distinction clear, we generally like to call $\wcone M$ the \emph{unified} warped cone. In this note we drop the ``unified'', as we are not going to need $\wc(\Gamma\curvearrowright M)$ in the sequel.
    \end{enumerate}
\end{remark*}

\medskip
\noindent\textbf{Acknowledgements.}
It is a pleasure to thank Ulrich Bunke for helpful comments on a previous version of this paper.
This work was funded by the  RTG 2491 -- \emph{Fourier Analysis and Spectral Theory} of the DFG.

\section{Background and positive results}\label{sec: Background}
A \emph{geometric module} for a proper metric space $X$ is a non-degenerate $\ast$-represen\-tation $\rho\colon C_0(X)\to\mathcal B(\mathcal H)$, where $\mathcal H$ is some separable (infinite dimensional) Hilbert space and $\mathcal B$ denotes the bounded operators. We will generally drop $\rho$ from the notation, and simply write $\mathcal H$ for the module.
It is \emph{ample} if $\rho(f)$ is not compact for any non-zero $f\in C_0(X)$.
If $M$ is a non-discrete Riemannian manifold, then $L^2(M)$ is an ample geometric module for $M$, where the representation is by pointwise multiplication.

An operator $t\in \mathcal B(\mathcal H)$ has \emph{finite propagation} if there is $R\geq 0$ such that $\rho(f)t\rho(g)=0$ for every $f,g\in C_0(X)$ with $d(\supp(f),\supp(g))>R$. It is \emph{locally compact} if $\rho(f)t$ and $t\rho(f)$ are compact for every $f\in C_0(X)$.

Let $\mathcal H$ be an ample module. We define:
\begin{align*}
    \Cfp{X} &\coloneqq  \overline{\braces{t\in\mathcal B(\mathcal H)\text{ of finite propagation}}},  \\
    \Clc{X} &\coloneqq \braces{t\in\mathcal B(\mathcal H)\text{ locally compact}}, \\
    \Croe{X} &\coloneqq \overline{\braces{t\in\mathcal B(\mathcal H)\text{ locally compact of finite propagation}}}.
\end{align*}
All the above are $C^*$-algebras: $\Cfp{X}$ and $\Croe{X}$ are defined taking the norm-closure to make them complete, while $\Clc{X}$ is already closed.
The local compactness in the definition of $\Croe{X}$ is necessary for this
$C^*$-algebra to have interesting K\=/theory. 

It is not hard to show that different choices of ample geometric modules
result in isomorphic $C^*$-algebras and that there is a \emph{canonical} class of
isomorphisms inducing a \emph{fixed} isomorphism of the K-groups. As we are
interested in the latter, we can safely drop the module from the notation.
More  generally, with every proper controlled function $f\colon X\to Y$ one can associate $\ast$-homomorphisms $\Croe{X}\to\Croe{Y}$ (here it is key that the modules be ample), which all induce the same homomorphism in K\=/theory $f_\ast\colon K_\ast(\Croe{X})\to K_\ast(\Croe{Y})$. The same is true for $\Cfp\variable$ and $\Clc{\variable}$ as well.
In particular, if $X$ and $Y$ are coarsely equivalent proper metric spaces then $\Cfp{X}\cong\Cfp{Y}$, $\Clc{X}\cong\Clc{Y}$, $\Croe{X}\cong\Croe{Y}$. We refer to \cites{MartinezVigolo2023RoelikeAlgebras,WillettYu2020HigherIndex} for details

In the following, we will still write $\Croe{\mathcal H}$ if we wish to stress that the module $\mathcal H$ is being used.

The coarse Baum--Connes
conjecture predicts that every element in $K_*(\Croe{X})$ is the index of a
K-homology class. As homology is intrinsically local, roughly speaking, the
coarse Baum--Connes conjecture implies that every element in $K_*(\Croe{X})$ has
``local flavor''. The description of K-homology as the K\=/theory of Yu's
localization algebra \cites{yu1997localization,qiao2010localization} shows that every
class in the image of $\mu_c$ can in fact be represented by operators of arbitrarily
small propagation.

An operator $t\in \mathcal B(\mathcal H)$ is \emph{ghost} if for every $\epsilon>0$ there is some compact subset $B\subseteq X$ such that $\norm{\rho(f)t\rho(g)}\leq \epsilon$ for any choice of $f,g\in C_0(X)$ of norm at most $1$ and supported on sets of diameter at most $1$ which are disjoint from $B$.\footnote{%
    If $X$ is a uniformly locally finite metric space and $\mathcal H = \ell^2(X;\ell^2(\mathbb N))$, $t\in\mathcal H$ is ghost if and only if for every $\epsilon>0$ there is $F\subset X$ finite such that the norm of all the coefficients $t_{xy}$ with $x,y\notin F$ is at most $\varepsilon$. This is how ``ghostness'' is usually defined.
    In the definition above, if $X$ has bounded geometry, the condition that the supports of $f,g$ have diameter at most $1$ can be equivalently replaced by ``have diameter at most $R$'' for any fixed choice of $R>0$.
}
Compact operators are easily seen to be ghost. In the converse direction, it is a natural (if somewhat naïve) guess that any ghost operator in the image of $\mu_c$ must be essentially supported on some compact subset of $X$ and thus be compact. To make this belief stronger, it is known that if $\Croe{X}$ does not contain any non-compact ghost operator, then $X$ has property A \cite{RoeWillett2014GhostbustingProperty}, and hence satisfies the coarse Baum--Connes Conjecture \cite{Yu2000CoarseBaum}. In particular, spaces for which $\mu_c$ is not an isomorphism \emph{must} have non-compact ghost operators in their Roe algebras. If moreover they contain a non-compact ghost \emph{projection} $P$ in the Roe algebra, it is then natural to expect the K\=/theory class $[P]\in K_0(\Croe{X})$ to be problematic.

\begin{example}\label{exmp: expanders}
    Let $X=\bigsqcup_{n\in\mathbb N} \mathcal G_n$ be a \emph{coarse disjoint union} of expander graphs (i.e.~the distance between vertices in two components $\mathcal G_n$, $\mathcal G_m$, $n\neq m$ is set to be some arbitrarily chosen value greater than the diameter of both $\mathcal G_n$, $\mathcal G_m$. The choice does not matter up to coarse equivalence). Let $P\in \ell^2(X)$ denote the projection onto the space of functions that are constant on each $\mathcal G_n$. The condition that $\mathcal G_n$ are a family of expanders implies that $P$ is a non-compact ghost that belongs to $\Croe{\ell^2(X)}$.
    One technical detail here is that $\Croe{\ell^2(X)}$ is \emph{not} the Roe algebra of $X$,  because $\ell^2(X)$ is not an ample module.
    The easy way out is then to fix some finite rank projection $q\in \mathcal B(\ell^2(\mathbb N))$ and observe that $P\otimes q\in \mathcal B(\ell^2(X,\ell^2(\mathbb N)))$ is now a non-compact ghost projection in $\Croe{\ell^2(X,\ell^2(\mathbb N))}\cong\Croe{X}$.

    Under a large girth assumption on the expander family, it can be shown that $[P\otimes q]$ lies outside the image of the coarse assembly map $\mu_c$ \cite{WillettYu2012HigherIndex}. Going beyond expanders, the same idea can be extended to several other classes of spaces $X$ described as coarse disjoint unions of compact subspaces that are ``expanding enough'' \cites{KhukhroLiVigolo2021StructureAsymptotic,LiSpakulaZhang2023MeasuredAsymptotic,LiVigoloZhang2023MarkovianRoealgebraic}.
    By and large, this is the main source of counterexamples to various variants of the coarse Baum--Connes conjecture \cites{AparicioJulgValette2019BaumConnes,WillettYu2020HigherIndex}.
\end{example}

Let now $\Gamma\curvearrowright M$ be an action by diffeomorphisms on a
compact Riemannian manifold, and $\wcone M$ the associated warped cone. The space $L^2(M\times \rpos) \cong L^2 M\otimes L^2\rpos$ is an ample module for $\wcone M$ when equipped with the natural representation by pointwise multiplication. Here we are of course giving $\rpos$ the Lebesgue measure.
Let $m\colon L^2M\to \mathbb C$ be the mean-value projection. Then $\mathfrak G \coloneqq m\otimes \id_{L^2\rpos}\in\mathcal B(L^2 M\otimes L^2\rpos)$ is the projection onto the space of functions that are constant on each level-set of the warped cone:
\[
\mathfrak G \xi (x,t)=\int_{M} \xi(y,t) \dd y.
\]
It is proved in \cite{DrutuNowak2019KazhdanProjections}*{Theorem 7.6} that if
the $\Gamma$-action is volume preserving and has a spectral gap, then
$\mathfrak G$ is approximable via finite propagation operators. Note the
\emph{crucial consequence} that then $\mathfrak G\in\Cfp{\wcone M}$.
This projection is clearly non-compact and is easily seen to be ghost. This motivated the statement of \cref{conj:roe}.

\begin{remark}
    The result of Dru\k{t}u--Nowak was greatly generalized in \cite{LiVigoloZhang2023MarkovianRoealgebraic}. Namely, \cite{LiVigoloZhang2023MarkovianRoealgebraic}*{Theorem E} shows that if $\Gamma\curvearrowright M$ is a measure-class preserving continuous action then $\mathfrak G\in\Cfp{\wcone M}$ if and only if $\Gamma\curvearrowright M$ is \emph{strongly ergodic}.\footnote{
      A measure-class preserving action on a probability space $\Gamma\curvearrowright (X,\nu)$ is strongly ergodic if every sequence of  almost invariant subsets (i.e.~measurable $C_n\subseteq X$ such that $\nu(C_n\triangle\gamma\cdot C_n)\to 0 $ for every $\gamma \in\Gamma$) must be almost invariant for trivial reasons (i.e.~$\nu(C_n)(\nu(X)-\nu(C_n))\to 0$).
    } 
\end{remark}

One minor technical difficulty at this point is that the projection $\mathfrak G$ never belongs to $\Croe{\wcone M}$, because it is quite clearly not locally compact. The easiest workaround is to replace the warped cone $\wcone M$ with the \emph{integral warped cone} $\iwcone M$, where we are using $\mathbb N+1$ in place of $\mathbb N_{\geq 1}$ for typesetting reasons (recall that we denote by $\wcone^A M$ the subspace $ M\times A\subseteq \wcone M$).
This is a coarsely dense subspace of the warped cone, and it is hence coarsely equivalent to it. It follows that $\Croe{\iwcone M}\cong\Croe{\wcone M}$, so we are entitled to work with the former.

So long as $M$ is not discrete (i.e.~it is not zero dimensional), the natural geometric module $L^2(M\times \npos)\cong L^2M\otimes\ell^2\npos$ is ample and can hence be used to construct the Roe algebra. The advantage now is that $\ell^2\npos$ is locally finite dimensional, hence the \emph{integral Dru\k{t}u--Nowak projection}
\[
    \mathfrak G^{\mathbb N+1}\coloneqq m\otimes \id_{\ell^2\npos}
\]
is always locally compact.

In \cite{Sawicki2017WarpedCones}, \cite{DrutuNowak2019KazhdanProjections}*{Theorem 7.6} is used to show that if $\Gamma\curvearrowright M$ is volume preserving and has a spectral gap then $\mathfrak G^{\mathbb N+1}$ is approximable via locally compact operators of finite propagation, and hence belongs to $\Croe{\iwcone M}$ (see also \cite{Sawicki2018GeometryMetric}*{Proposition 9.3}). This fact was extended to arbitrary strongly ergodic actions in \cite{LiVigoloZhang2023MarkovianRoealgebraic}*{Proposition 5.1}. In particular, in this setup the Roe algebra does contain the desired non-compact ghost projections. However, using $\mathfrak G^{\mathbb N+1}$ to prove that $\mu_c$ is not an isomorphism turned out to be much more delicate than anticipated.

\smallskip

There is one last result worth mentioning here. Namely, if one further restricts to the subspace $\wcone^{2^{\mathbb N}} M \subset \iwcone M$ (here and in the sequel $2^\naturals\coloneqq\braces{2^n\mid n\in \mathbb N}$)
then one can show that $\wcone^{2^{\mathbb N}} M$ does violate the coarse Baum--Connes Conjecture.
To see this, consider the associated Dru\k{t}u--Nowak projection $\mathfrak G^{2^{\mathbb  N}}\coloneqq m\otimes \id_{2^{\mathbb N}}$. Using the techniques of \cite{WillettYu2012HigherIndex}, the following can be shown:

\begin{theorem}[\cite{Sawicki2017WarpedCones}*{Theorem 3.5}, \cite{LiVigoloZhang2023MarkovianRoealgebraic}*{Theorem G}]\label{thm: old-counterexample}
 If $\Gamma\curvearrowright M$ is a free, strongly ergodic, Lipschitz action and $\Gamma$ has property A, then the K\=/theory class $[\mathfrak G^{2^{\mathbb  N}}]$ is non-zero and lies outside the image of the coarse assembly map $\mu_c$
\end{theorem}
\begin{proof}[Idea of proof]
  Considering the restrictions of operators to the level sets $M\times\{2^n\}$ yields a $\ast$-homomorphism of $\Croe{\wcone}$ into $\prod_{\expnat}\mathcal K(L^2M)/\bigoplus_\expnat \mathcal K(L^2M)$, where $\mathcal K$ denotes the compact operators. Composing this with the canonical trace and taking K\=/theory results in a homomorphism
  \[
  \tau_{\rm d}\colon K_0(\Croe{\wcone})\to \frac{\prod_\expnat \mathbb R}{\bigoplus_\expnat \mathbb R}.
  \]
  The class of the Dru\k{t}u--Nowak projection $\mathfrak G^{2^{\mathbb  N}}$ is mapped to $[(1,1,1,\ldots)]$, that is $\tau_{\rm d}([\mathfrak G^{2^{\mathbb  N}}])=[(1,1,1,\ldots)]\neq 0$.
  
  Freeness, the Lipschitz condition and property A can be used to define another trace $\tau_{\rm u}$ that vanishes on ghost projections, and hence on $[\mathfrak G^{2^{\mathbb  N}}]$. However, it can be shown that $\tau_{\rm d}$ and $\tau_{\rm u}$ coincide on the image of the coarse assembly map.
\end{proof}

In the above, there is nothing special about the set $\expnat$. What one really needs is to consider a set of the form $A=\bigsqcup_{n\in\mathbb N} A_n$ that is a union of bounded non-empty sets with $d(A_n,A_m)\to\infty$ \cite{Sawicki2018GeometryMetric}. Namely, for those techniques to work it
is necessary that $\wcone^AM\subset \wcone M$ is a coarse disjoint union of
bounded collections of level sets. This implies that we are once again in a situation analogous to that of \cref{exmp: expanders}, which is a somewhat underwhelming result. It would have been much better to show that $[\mathfrak G^{\mathbb N+1}]$ is not in $\im(\mu_c)$. However, we shall presently see that, in fact, $[\mathfrak G^{\mathbb N+1}]=0$.

\section{Vanishing of K-class of  projections}\label{sec: vanishing proof}

To prove the vanishing result \cref{thm:vanishing}, it is helpful to slightly recast the construction of the integral Dru\k{t}u--Nowak projection $\mathfrak G^{\mathbb N+1}$ working directly on $\wcone M$, and avoiding to pass to the space of integer level sets.
For every $n\in\npos$,  choose some $\xi_n\in L^2\mathbb R$ of norm one and supported on $[n,n+1]$. Note that these functions are pairwise orthogonal. Let $q\in \mathfrak B(L^2\rpos)$ be the projection on their closed span $\overline{\angles{\xi_n\mid n\in \npos}}$.

Observe now that the image of $\id_{L^2 M}\otimes q$ can be naturally identified with \linebreak $L^2M\otimes\ell^2(\npos)$. 
In other words, choosing $\xi_n$ as above defines an isometric embedding $L^2M\otimes\ell^2(\npos)\hookrightarrow L^2(M\times \rpos)$. 
In turn, this yields a $\ast$-embedding  $\mathcal B(M\otimes\ell^2\npos)\hookrightarrow \mathcal B(L^2(M\times \rpos))$, which is easily seen to map $\Croe{\iwcone M}$ into $\Croe{\wcone M}$.

\begin{remark}
    A more sophisticated way of phrasing this is that $\id_{L^2 M}\otimes q$ is a submodule of $L^2(M\times\rpos)$, whose associated Roe algebra is naturally identified with $\Croe{\iwcone M}$ \cite{MartinezVigolo2024RigidityFramework}.
\end{remark}

Now, as explained in \cref{sec: Background}, one can show that if
$\Gamma\curvearrowright M$ is strongly ergodic, then the integral Dru\k{t}u--Nowak projection is in $\Croe{\iwcone M}$. In turn, its image under the embedding of $\Croe{\iwcone M}$ into $\Croe{\wcone M}$ is a non-compact ghost projection in $\Croe{\wcone M}$, which is easily seen to be nothing but the projection $m\otimes q$.

\medskip

Up to this point we have done nothing new. There is however a better argument to show that $m\otimes q$ belongs to the Roe algebra. In \cites{Sawicki2017WarpedCones,LiVigoloZhang2023MarkovianRoealgebraic} it takes some effort to show that $\mathfrak G^{\mathbb N+1}$ belongs to $\Croe{\iwcone M}$ by explicitly constructing an approximation via locally compact finite propagation operators (i.e.~using the definition of Roe algebra directly). On the other hand, it is now known that the Roe algebra of a proper metric space can also be defined as the intersection
\begin{equation}\label{eq: roe as intersection}
    \Croe{X} =\Cfp{X}\cap \Clc{X}
\end{equation}
(one containment is obvious, the other is proved in \cite{BragaVignati2023GelfandtypeDualityxo}*{Proposition 2.1} and \cite{MartinezVigolo2023RoelikeAlgebras}*{Theorem 6.20}). This fact can be combined with the following simple lemma.

\begin{lemma}\label{lem: tensors in Roe}
    Let $t\in \mathcal B(L^2 M)$ and $s\in\mathcal B(L^2\rpos)$.
    \begin{enumerate}
        \item\label{item: product} If $t\otimes\id_{L^2\rpos}$ and $\id_{L^2 M}\otimes s$ belong to $\Cfp{\wcone M}$, then $t\otimes s\in \Cfp{\wcone M}$.
        \item\label{item: fp} If $s\in \Cfp{L^2\rpos}$, then $\id_{L^2 M}\otimes s\in \Cfp{\wcone M}$.
        \item\label{item: lc} If $t\in \mathcal K(L^2 M)$ and $s\in\Clc{L^2\rpos}$, then $t\otimes s\in\Clc{\wcone M}$.
    \end{enumerate}
\end{lemma}
\begin{proof}
    For \eqref{item: product} it suffices to write $t\otimes s = (t\otimes\id_{L^2\rpos})(\id_{L^2 M}\otimes s)$.
    \eqref{item: fp} follows easily from the definition of the warped metric.
    For \eqref{item: lc}, any function $f\in C_0(\wcone M)$ is the limit of
    the functions $f_n\coloneqq (1\otimes g_n)f$, where $g_n\in C_0(\rpos)$ is
    chosen to be equal to $1$ on $[1,n]$. Then
    \[
        (t\otimes s)\rho(f)
        =\lim_{n\in \mathbb N}(t\otimes s)\rho(1\otimes g_n)\rho(f)
        =\lim_{n\in \mathbb N}(t\otimes s\rho(g_n))\rho(f)
    \]
    is compact, and the same applies to $\rho(f)(t\otimes s)$ as well.
\end{proof}

\begin{corollary}\label{cor: Dr-No are in Roe}
    If $p\in \mathcal K(L^2M)$, $q\in \Croe{\rpos}$ and $p\otimes\id_{L^2\rpos}\in\Cfp{\wcone M}$, then $p\otimes q\in \Croe{\wcone M}$.
\end{corollary}

If $q$ is the projection onto $\overline{\angles{\xi_n\mid n\in \npos}}$ as above, then $q\in \Croe{\rpos}$.
The condition $p\otimes\id_{L^2\rpos}\in\Cfp{\wcone M}$ is exactly what was
verified by Dru\k{t}u--Nowak for $\mathfrak G = m\otimes\id_{L^2\rpos}$ under
the spectral gap assumption, and the above then implies that the integral
Dru\k{t}u--Nowak projection $m\otimes q$ does indeed belongs to the Roe
algebra.

It is now worth spending a few words on the proof that
$m\otimes\id_{L^2\rpos}$ is in $\Cfp{\wcone M}$. An operator $t \in \mathcal
B(L^2M)$ is said to have \emph{finite dynamical propagation} if there is some
$R\geq 0$ such that $\rho(f)t\rho(g)=0$ for every $f,g\in C(M)$ with $
supp(f)\cap \gamma\cdot\supp(g)=\emptyset$ for every $\gamma\in \Gamma$ of
word-length at most $R$
\cites{LiVigoloZhang2023MarkovianRoealgebraic,deLaatVigoloWinkel2023DynamicalPropagation}.
Define:
\[
\Cfp{\Gamma\curvearrowright M}\coloneqq\overline{\braces{t\in \mathcal B(L^2 M)\text{ of finite dynamical propagation}}}.
\]
It is clear from the definition of the warped metric that if $t$ has dynamical propagation at most $R$, then $t\otimes\id_{L^2\rpos}$ has propagation at most $R$ in $\wcone M$. In particular, the mapping $t\mapsto t\otimes\id_{L^2\rpos}$ defines a $\ast$-embedding
\[
\Cfp{\Gamma\curvearrowright M}\hookrightarrow \Cfp{\wcone M}.
\]
It is shown in \cites{LiVigoloZhang2023MarkovianRoealgebraic} that if $\Gamma\curvearrowright M$ is a continuous action, then it is strongly ergodic if and only if $m\in \Cfp{\Gamma\curvearrowright M}$.

Inspired by this discussion, we make the following:

\begin{definition}\label{def: generalized Dr-No}
    A \emph{generalized Dru\k{t}u--Nowak projection} is a non-compact projection of the form $p\otimes q$, where $p\in \Cfp{\Gamma\curvearrowright M}\cap \mathcal K(L^2M)$ and $q\in \Croe{\rpos}$.
\end{definition}

By \cref{cor: Dr-No are in Roe}, every generalized Dru\k{t}u--Nowak projection is a non-compact ghost projection in the Roe algebra of $\wcone M$, and \cites{DrutuNowak2019KazhdanProjections,LiVigoloZhang2023MarkovianRoealgebraic} show that every strongly ergodic action gives rise to generalized Dru\k{t}u--Nowak projections.

It is not hard to see and proven in \cite{deLaatVigoloWinkel2023DynamicalPropagation} that if $\Gamma\curvearrowright M$ is an ergodic action then $\Cfp{\Gamma\curvearrowright M}$ is an irreducible $C^*$-algebra (i.e.~there are no non-trivial $\Cfp{\Gamma\curvearrowright M}$-invariant closed subspaces of $L^2 M$).
On the other hand, it is well known that if $A\leq \mathcal B(\mathcal H)$ is an irreducible $C^*$-algebra such that $\mathcal K(\mathcal H)\cap A\neq \{0\}$, then $\mathcal K(\mathcal H)\subseteq A$. With this at hand, it is now clear how to prove \cref{thm:vanishing}:

\begin{proof}[Proof of \cref{thm:vanishing}]
    Let $p\otimes q$ be a generalized Dru\k{t}u--Nowak projection. By assumption, $\Cfp{\Gamma\curvearrowright M}\cap \mathcal K(L^2M)\neq \{0\}$. Since the action is ergodic, $\Cfp{\Gamma\curvearrowright M}$ is irreducible, and hence $\mathcal K(L^2M)\leq \Cfp{\Gamma\curvearrowright M}$.
    By \cref{cor: Dr-No are in Roe}, it then follows that $\mathcal K(L^2M)\otimes\Croe{\rpos}$ is contained in $\Croe{\wcone M}$.
    
    On the other hand, tensoring with the compacts does not change the K\=/theory, hence
    \[
    K_*(\mathcal K(L^2M)\otimes\Croe{\rpos}) = K_*(\Croe{\rpos}) ,
    \]
    and the latter is $\{0\}$ because $\rpos$ is a flasque space \cite{Roe1996IndexTheory}*{Proposition 9.4}.
    Since $p\otimes q \in \mathcal K(L^2M)\otimes\Croe{\rpos}$, and by
    naturality of K\=/theory of $C^*$-algebras, it follows that
    \begin{equation*}
      [p\otimes q]\in
      K_0(\Croe{\wcone M})
    \end{equation*}
    lies in the image of $K_0(\mathcal K(L^2(M)\otimes \Croe{\rpos}))=\{0\}$ and hence
    vanishes.
\end{proof}

\begin{corollary}
    If $\Gamma\curvearrowright M$ is an action as by \cref{conj:roe}, then $\mathfrak G^{\mathbb N+1}$ vanishes in K\=/theory, and hence belongs to the image of the coarse assembly map.
\end{corollary}

\subsection{Proof of vanishing by a direct Eilenberg swindle}\label{ssec: swindle}

After realizing that $[\mathfrak G]$ vanishes in $K_0(\Croe{\wcone M})$, it is not hard to find several proofs of this fact. For instance, we can show it using an Eilenberg swindle argument similar to the proof of \cite{WillettYu2020HigherIndex}*{Proposition 7.5.2}.

In order to do so, we begin by choosing a very ample geometric module.
Namely, we set $\mathcal{H}_\infty \coloneqq \bigoplus_{n\in\naturals}
L^2(M\times \npos)$.
For every $n\in \naturals$, we also let
\[W_n \colon L^2(M\times \npos)  \longrightarrow \mathcal{H}_{\infty}\] to be
the isometry mapping $L^2(M\times \npos)$ into the $n$-th summand. Both
$L^2(M\times \npos)$ and $\mathcal H_\infty$ are ample geometric modules for
$\wcone^\naturals M$, and can hence be used to construct the Roe algebra of
$\wcone^\naturals M$. Moreover the conjugation $\Ad_{W_0}(s)\coloneqq
{W_0}s{W_0}^*$ induces an isomorphism in K-theory
\[
(\Ad_{W_0})_\ast \colon K_\ast(\Croe{L^2(M\times \npos)})\xrightarrow{\ \cong\ } K_\ast(\Croe{\mathcal H_\infty})
\]
because $W_0$ is an isometry covering the identity map of $\wcone^\naturals M$.

Let
\(
S_n \colon L^2(M\times\npos)\to L^2(M\times\npos)
\)
denote the isometry defined by the shift
\[
(S_n\xi)(x,t)\coloneqq
\left\{
\begin{array}{ll}
  \xi(x, t - n) & t\geq n\\
  0 & \text{else}.
\end{array}
\right.
\]
We consider the $\ast$-homomorphism $\Phi\colon \mathcal B(L^2(M\times\npos))\to
\mathcal B(\mathcal H_\infty)$ given by the sum
\[
\Phi\coloneqq \sum_{n\in\naturals} \Ad_{W_nS_n}.
\]

It is emphatically not the case that $\Phi$ restricts to a homomorphism of Roe
algebras, because $\Ad_{S_n}$ will increase the propagation as $n$ grows.
However, if $\mathfrak G^\naturals$ is in $\Croe{L^2(M\times\npos)}$ (i.e.\ the action is strongly ergodic), then $\Phi(\mathfrak
G^\naturals)$ is in $\Croe{\mathcal H_\infty}$. This is once again very easy to
see using the description of the Roe algebra as an intersection \eqref{eq: roe
as intersection}, but it can also be verified by hand: suppose that $\mathfrak
G^\naturals = \lim_{k\to\infty}s_k$ with $s_k\in\mathcal B(L^2(M\times\npos))$
locally compact and of finite propagation. Let $s_k^{(n)}\in \mathcal B(\mathcal
H_\infty)$ to be the operator that coincides with $W_ns_kW_n^*$ on
$\bigoplus_\naturals L^2(M\times \naturals_{\geq n})$ and that is zero on
$\bigoplus_\naturals L^2(M\times \naturals_{< n})$. Then we see that
\[
\Phi(\mathfrak G^\naturals) = \lim_{k\to\infty}\left(\sum_{n\in\naturals} s_k^{(n)}\right),
\]
and the infinite sums on the RHS define locally compact operators of finite
propagation for every $k\in \naturals$.

At this point we are essentially done. In fact, we observe that
\[
\Phi(\mathfrak G^\naturals) = \mathfrak G^\naturals + \Ad_V(\Phi(\mathfrak G^\naturals)),
\]
where $V\colon\mathcal H_\infty\to\mathcal H_\infty$ is the isometry sending a
function $\xi$ contained in the $n$\=/th copy of $L^2(M\times\npos)$ to the
shifted function $S_1(\xi)$ in the $(n+1)$-th copy of $L^2(M\times\npos)$. Since
$V$ covers the identity of $\wcone^\naturals M$, we deduce that
\[
 [\Phi(\mathfrak G^\naturals)] = [\mathfrak G^\naturals] + [\Ad_V(\Phi(\mathfrak G^\naturals))]
 = [\mathfrak G^\naturals] + [\Phi(\mathfrak G^\naturals)],
\]
hence $[\mathfrak G^\naturals]$ must be zero.

\begin{remark}
  Ulrich Bunke has explained to us a ``motivic'' approach to prove a general
  vanishing result for projections. It relies entirely on properties of the
  semi-additive category of coarse spaces with transfer (for branched
  covering projections) of \cite{BEKWtr} and therefore holds for every coarse
  cohomology theory with transfer.
 The Eilenberg-swindle argument we outlined above can be seen as a special case of this more general approach.
\end{remark}

\section{A Mayer--Vietoris argument}\label{sec: Mayer--Vietoris}

\subsection{Setup}\label{ssec:coarse homology setup}
The main technical computation of this paper is a Mayer--Vietoris argument for coarse homology theories. There are several ways of formalizing what a coarse homology theory is, see e.g.~\cites{mitchener2001coarse,wulff2022equivariant,BunkeEngel2020HomotopyTheory}. Rather than committing to one definition over another, we list below the axioms that we will need for our computation.
We need $HX_\ast$ to be a functor from the category of proper metric spaces and controlled proper functions to the category of graded abelian groups such that the following hold:
\begin{itemize}
  \item $HX_\ast([0,\infty))=\{0\}$ (flasqueness axiom);
  \item If $X = A\cup B$ is a \emph{(coarsely) excisive pair} then there is a natural Mayer--Vietoris long exact sequence
  \[
  \begin{tikzcd}
    \cdots \ar{r} &[-1.5em]
    HX_\ast(A\cap B)\ar{r}{({\iota_A}_\ast,{\iota_B}_\ast)}
    &[0em]
    HX_\ast(A)\oplus HX_\ast(B) \ar{r}{{j_A}_\ast - {j_B}_\ast}
    \ar[draw=none]{d}[name=X, anchor=center]{}
    &[0em]
    HX_\ast(X)
    \ar[rounded corners,
      to path={ -- ([xshift=2ex]\tikztostart.east)
        |- (X.center) \tikztonodes
        -| ([xshift=-2ex]\tikztotarget.west)
        -- (\tikztotarget)}]{dll}[at end]{\partial}
    \\
    &
    HX_{\ast-1}(A\cap B)\ar{r}
    & \cdots\hspace{10em}
    &
  \end{tikzcd}
  \]
  \item If $f,g\colon X\to Y$ are \emph{coarsely homotopic} then $f_\ast=g_\ast$.
\end{itemize}
In the above, $X = A\cup B$ is \emph{excisive} if for every $r>0$ there is some $R>0$ such that $N_r(A)\cap N_r(B)\subseteq N_R(A\cap B)$, where $N_r$ denotes the $r$-neighbourhood. In more sophisticated words, the pair $A,B$ is excisive if $A\cap B$ is (a representative of) their coarse intersection \cite{LeitnerVigolo2023InvitationCoarse}.

More delicate is the definition of coarse homotopy, as there are several such notions in the literature. For instance, in \cite{higson1994homotopy}*{Definition 1.2} two (continuous) maps $f,g$ are called coarsely homotopic if:

\noindent\begin{minipage}{4em}
  \centering
  $(\star)$
\end{minipage}\begin{minipage}{\textwidth -4em}
  \medskip
  there is a continuous proper map $h\colon X\times [0,1]\to Y$ such that
  $h$ restricts to $f$ and $g$ at $s=0$ and $s=1$, respectively, and 
  the sections $h(\variable,s)$ are equicontrolled as $s\in[0,1]$ varies (i.e.~for every $r$ there is $R$ such that $d(x,x')\leq r\implies d(h(x,s),h(x',s))\leq R$).
  \medskip
\end{minipage}
(this definition is cleverly extended in  \cite{wulff2022equivariant}*{Definition 2.16}).
Another definition is that of \cite{BunkeEngel2020HomotopyTheory}*{Definition 4.17} (which extends \cite{HigsonPedersenRoe1997AlgebrasControlled}*{Definition 11.1}). There, $f,g$ are called coarsely homotopic if:

\noindent\begin{minipage}{4em}
  \centering
  $(\star\star)$
\end{minipage}\begin{minipage}{\textwidth -4em}
  \medskip
  there are bornological functions $\rho_\pm\colon X\to [0,\infty)$ and a controlled map $H\colon X\times [0,\infty)\to Y$ such that
  \begin{itemize}
    \item if $s\leq\rho_-(x)$ (resp.\ $s\geq\rho_+(x)$) then $H(x,s)=f(x)$ (resp.\ $H(x,s)=g(x)$)
    \item the restriction of $H$ to $\{(x,s)\mid \rho_-(x)\leq s\leq \rho_+(x)\}$ is proper.
  \end{itemize}
  \medskip
\end{minipage}
Several other variants of these definitions can be found in the literature,
compare e.g.~\cite[Section 2]{MNS}. Fortunately, the homotopies we are going to need (\cref{lem:coarse homotopy equivalence}) are easily seen to be coarse homotopies with respect to any of those definitions.

\subsection{Mayer--Vietoris on warped cones}
For the rest of this section the action $\Gamma\curvearrowright M$ is fixed. For
ease of notation, we thus write $\wcone^A$ in place of $\wcone^A M$.
We observe the following:

\begin{lemma}\label{lem:excisive is excisive}
  If $A = A_1\cup A_2\subseteq \rpos$ is an excisive pair, then so is $\wcone^A = \wcone^{A_1}\cup\wcone^{A_2}$.
\end{lemma}
\begin{proof}
  Follows directly from the definition of the warped metric.
\end{proof}

We are going to carry out a Mayer--Vietoris computation on the warped cone by splitting it along the level sets $\expnat$. Let
\begin{align*}
      I_1 &:= \bigsqcup_{n\in\naturals} [2^{2n}, 2^{2n+1}] \\
      I_2 &:= \bigsqcup_{n\in\naturals} [2^{2n+1},2^{2n+2}]
\end{align*}
Note that $I_1\cup I_2=\rpos$ and $I_1\cap I_2= 2^{\naturals+1}$. Correspondingly,
\begin{equation*}
  \wcone^{I_1}\cup \wcone^{I_2} = \wcone ;
  \qquad
  \wcone^{I_1}\cap \wcone^{I_2} = \wcone^{2^{\naturals+1}}.
\end{equation*}

Let $\text{od}\coloneqq 2\mathbb N +1$ and $\text{ev}\coloneqq 2\mathbb N +2$.
Then $\wcone^\expodd$ and $\wcone^\expeven$ embed in both $\wcone^{I_1}$ and
$\wcone^{I_2}$. In turn, both $\wcone^{I_1}$ and $\wcone^{I_2}$ can be collapsed
onto either of them. We name these mappings as follows:
\[
\begin{tikzcd}
  \wcone^\expodd 
  \ar[-,double line with arrow={<-,->}, label={asd}]{r}[yshift=2pt]{\topp_1}[swap,yshift=-2pt]{\up_1}
  & \wcone^{I_1} 
  & \wcone^\expodd 
  \ar[-,double line with arrow={<-,->}, label={asd}]{r}[yshift=2pt]{\bott_2}[swap,yshift=-2pt]{\down_2}
  & \wcone^{I_2} \\
  \wcone^\expeven 
  \ar[-,double line with arrow={<-,->}, label={asd}]{r}[yshift=2pt]{\bott_1}[swap,yshift=-2pt]{\down_1}
  & \wcone^{I_1} 
  & \wcone^\expeven 
  \ar[-,double line with arrow={<-,->}, label={asd}]{r}[yshift=2pt]{\topp_2}[swap,yshift=-2pt]{\up_2}
  & \wcone^{I_2}.
\end{tikzcd}
\]
Namely, `$\topp$' and `$\bott$' represent the inclusion as top or bottom
extremity, respectively, of the intervals, whereas `$\up$' and `$\down$' map a point $(x,t)$
with $t\in [2^k,2^{k+1}]$ to $(x,2^{k+1})$ and $(x,2^k)$, respectively, see \cref{fig}. There is a small
caveat in that `$\down$' does not map the tip of $\wcone^{I_1}$ into
$\wcone^\expeven$. We
fix this by letting $(x,t)\mapsto (x,4)$ for every $t\in [1,2]$.

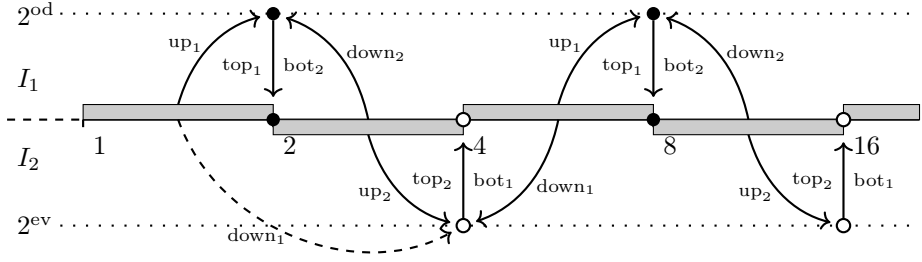
\begin{figure}
  \centering
  \begin{tikzpicture}[scale=1, thick]
      \def\hgt{4em}
      \def\shgt{1.5em}
      \draw[dashed](-1,0)--(0,0);
      \draw (-1,\shgt)node[right]{$I_1$}
            (-1,-\shgt)node[right]{$I_2$}
            (-1,\hgt)node[right]{$\expodd$}
            (-1,-\hgt)node[right]{$\expeven$};
      \draw[] (0,0) -- (11,0);
      \draw[loosely dotted](-0.3,\hgt) -- (11,\hgt);
      \draw[loosely dotted](-0.3,-\hgt) -- (11,-\hgt);
      \foreach \x/\pos in {1/0, 2/2.5, 4/5, 8/7.5, 16/10} {
          \draw (\pos,0) -- ++(0,-0.1) node[below right] {\x};
      }
      
      \draw[thin,fill=black!20] (0,0) rectangle (2.5,0.2);  
      \draw[thin,fill=black!20] (5,0) rectangle (7.5,0.2);  
      \draw[thin,fill=black!20] (10,0) rectangle (11,0.2);  
      
      \draw[thin,fill=black!20] (2.5,-0) rectangle (5,-0.2);  
      \draw[thin,fill=black!20] (7.5,-0) rectangle (10,-0.2); 
      
      \foreach \x in {2.5,7.5} {
        \draw[->] (\x,\hgt-2.5pt) -- (\x,0.3)
            node[pos=0.6,left]{{\scriptsize$\topp_1$}}
            node[pos=0.6,right]{{\scriptsize$\bott_2$}};
        \fill (\x,0) circle[radius=2.5pt];
        \fill (\x,\hgt) circle[radius=2.5pt];
      }
      
      \foreach \x in {5,10} {
        \draw[->] (\x,-\hgt+2.5pt) -- (\x,-0.3)
          node[pos=0.5,left]{{\scriptsize$\topp_2$}}
          node[pos=0.5,right]{{\scriptsize$\bott_1$}};
        \draw[fill=white] (\x,0) circle[radius=2.5pt];
        \draw[fill=white] (\x,-\hgt) circle[radius=2.5pt];
      }
      
      \draw[->, shorten >= 5pt,bend left]  (1.25,0.2) to node[left]{{\scriptsize $\up_1$}}  (2.5,\hgt);
      \draw[->, dashed, shorten >= 5pt,bend right=50]  (1.25,0) to node[left]{{\scriptsize $\down_1$}}  (5,-\hgt);

      \draw[->, shorten >= 5pt,bend left]  (6.25,0.2) to node[left]{{\scriptsize $\up_1$}} (7.5,\hgt);
      \draw[->, shorten >= 5pt,bend left]  (6.25,0) to node[right]{{\scriptsize $\down_1$}} (5,-\hgt);

      \draw[->, shorten >= 5pt,bend right] (3.75,-0.2) to node[left]{{\scriptsize $\up_2$}}  (5,-\hgt) ;
      \draw[->, shorten >= 5pt,bend right] (3.75,-0) to node[right]{{\scriptsize $\down_2$}}  (2.5,\hgt) ;

      \draw[->, shorten >= 5pt,bend right] (8.75,-0.2) to node[left]{{\scriptsize $\up_2$}}  (10,-\hgt);
      \draw[->, shorten >= 5pt,bend right] (8.75,-0) to node[right]{{\scriptsize $\down_2$}}  (7.5,\hgt);
  \end{tikzpicture}
  \caption{Mappings among intervals and their extremities. All the solid arrows represent 1- or 2-Lipschitz maps, while the dashed one is 4-Lipschitz.}\label{fig}
\end{figure}

\begin{lemma}\label{lem:coarse homotopy equivalence} The above pairs of maps are
  coarse homotopy inverses to one another. In particular, $\topp_\ast$ and
  $\bott_\ast$ are isomorphisms with inverses $\up_\ast$ and $\down_\ast$
  respectively.
\end{lemma}
\begin{proof}
  Of course, $\up_1\circ \topp_1$ is the identity and $\topp_1\circ \up_1$ is homotopic to the identity by linearly varying the height.
  Explicitly, defining $h\colon \wcone^{I_1}\times[0,1]\to\wcone^{I_1}$ by $h((x,t),s)\coloneqq (x, st + (1-s)2^{2n+1})$ for $2^{2n}\leq t\leq 2^{2n+1}$ yields a coarse homotopy in the sense $(\star)$.

  Defining $H$ on $X\times 2^{[2n,2n+1]}\times \mathbb R$ by
  \[
    H((x,t),s)\coloneqq\left\{ 
      \begin{array}{ll}
        (x, (2^{-2n}s)t + (1-2^{-2n}s)2^{2n+1}) & s\in [0,2^{2n}]\\
        (x, t) & s > 2^{2n}
      \end{array}
    \right.
  \] 
  yields a coarse homotopy in the sense $(\star\star)$. The other cases are analogous.
\end{proof}

Note that the ``shift upwards of the level sets'' map gives a coarse equivalence
\begin{equation*}
  \begin{tikzcd}[column sep = 0pt,row sep = 0pt]
    \shift\colon &[-1em] \wcone^{\expnat}\ar{r} &[2em] \wcone^\expnat;
    \\ 
    & (x,t)\ar[|->]{r} &(x,2t).
  \end{tikzcd}  
\end{equation*}
Writing $2^{\mathbb N+1}=\expodd\sqcup\expeven$ (and ignoring bounded issues at
the tip of the cone which are irrelevant in the coarse category), we observe that $\shift$ swaps $\wcone^\expodd$ with
$\wcone^\expeven$ and that we have $\shift = \up\circ\bott$. In coarse homology,
let
\begin{align*}
  s_{\rm od}\coloneqq (\shift|_{\wcone^\expodd})_\ast\colon
  & HX_\ast(\wcone^\expodd)\to HX_\ast(\wcone^\expeven),
  \\
  s_{\rm ev}\coloneqq (\shift|_{\wcone^\expeven})_\ast\colon
  & HX_\ast(\wcone^\expeven)\to HX_\ast(\wcone^\expodd).
\end{align*}
The following is then immediate.

\begin{lemma}\label{lem:shift in homology} We have:
  \begin{align*}
    s_{\rm od}=& (\topp_2)_\ast^{-1}\circ(\bott_2)_\ast\\
    s_{\rm ev} =& (\topp_1)_\ast^{-1}\circ(\bott_1)_\ast.
  \end{align*}
\end{lemma}
\begin{proof}
 As noted above, $s_{\rm od}\coloneqq (\shift|_{\wcone^\expodd})_\ast = (\up_2\circ\bott_2)_\ast$. We see by \cref{lem:coarse homotopy equivalence} that $(\up_2)_\ast = (\topp_2)_\ast^{-1}$, from which the claim follows. The same applies to $s_{\rm ev}$ as well.
\end{proof}

Let now $\iotaodd\colon\wcone^\expodd\to\wcone^\expnat$ and
$\iotaeven\colon\wcone^\expeven\to\wcone^\expnat$ denote the inclusions.
A first application of Mayer--Vietoris gives the following.

\begin{lemma}\label{lem:M-V on naturals}
In the situation of \cref{thm:LES}, coarse Mayer--Vietoris yields a natural
short exact sequence 
  \[
    \{0\}\to
    HX_\ast(\pt)\xrightarrow{(\iota_\ast,-\iota_\ast)}
    HX_\ast\bigparen{\wcone^\expodd}\oplus HX_\ast\bigparen{\wcone^\expeven}
    \xrightarrow{\hspace{-.5ex}(\iotaodd)_\ast+(\iotaeven)_\ast\hspace{-.5ex}}HX_\ast\bigparen{\wcone^\expnat}
    \to\{0\}.
  \]
\end{lemma}
\begin{proof}
  The inclusions $\wcone^\expodd\hookrightarrow\wcone^{\expodd\cup\{1\}}$ and
  $\wcone^\expeven\hookrightarrow\wcone^{\expeven\cup\{1\}}$ are coarse
  equivalences, hence induce natural isomorphisms in coarse homology. Note that
  $\expnat= (\expodd\cup\{1\})\cup(\expeven\cup\{1\})$ is an excisive pair.
  By \cref{lem:excisive is excisive}, we can apply coarse Mayer--Vietoris to obtain a long exact sequence 
  \[
  \begin{tikzcd}
    \cdots\ar{r} &
    HX_\ast(\wcone^{\{1\}})\ar{r}{(\iota_\ast,-\iota_\ast)} &[2ex]
    HX_\ast\bigparen{\wcone^{\expodd\cup\{1\}}}\oplus HX_\ast\bigparen{\wcone^{\expeven\cup\{1\}}}
    \ar{d}{(\iotaodd)_\ast+(\iotaeven)_\ast}
    & \\
    && HX_\ast(\wcone^{\expnat})\ar{r}{\partial}&\cdots
  \end{tikzcd}
  \]

  Observe that
  $HX_\ast(\wcone^{\{1\}})=HX_\ast(\pt)$, because $\wcone^{\{1\}}$ is
  bounded. Moreover, the projection $\pi\colon \wcone^{\expodd\cup\{1\}}\to\expnat$
  sending $(x,2^{k})\mapsto 2^{k}$ is proper and controlled, hence induces a map in $HX_\ast$.
  The composition of $\iota$ with $\pi$ is the
  inclusion $\{1\}\hookrightarrow\expnat$ which is injective in $HX_*$ by
  hypothesis. Hence the arrow $(\iota_\ast,-\iota_\ast)$ is injective and by exactness
  all the boundary maps in the long exact sequence vanish, which proves the lemma.
\end{proof}

Note that the diagram
\[
\begin{tikzcd}[ampersand replacement=\&, column sep=5em, row sep = 2em]
  HX_\ast\bigparen{\wcone^{\expodd}}\oplus HX_\ast\bigparen{\wcone^{\expeven}}
  \ar[->>]{d}[swap]{(\iotaodd)_\ast+(\iotaeven)_\ast}
  \ar{r}{
    \begin{psmallmatrix}
      0 & s_{\rm ev} \\
      s_{\rm od} & 0
    \end{psmallmatrix}
  }
  \&
  HX_\ast\bigparen{\wcone^{\expodd}}\oplus HX_\ast\bigparen{\wcone^{\expeven}}
  \ar[->>]{d}{(\iotaodd)_\ast+(\iotaeven)_\ast}.
  \\
  HX_\ast(\wcone^{\expnat})\ar{r}{\shift_\ast}
  \&
  HX_\ast(\wcone^{\expnat})
\end{tikzcd}
\]
commutes. We now have all that it takes to prove \cref{thm:LES}.

\begin{proof}[Proof of \cref{thm:LES}]
  The decomposition $\rpos = I_1\cup I_2$ is excisive, hence we obtain a
  Mayer--Vietoris exact sequence
 \begin{equation} \label{eq:M-V on I1, I2}
    \cdots \to
    HX_\ast(\wcone^\expnat)
    \xrightarrow{({\iota_1}_*,{\iota_2}_*)}
    HX_\ast\bigparen{\wcone^{I_1}}\oplus HX_\ast\bigparen{\wcone^{I_2}}
    \xrightarrow{{j_1}_* -{j_2}_*}
    HX_\ast(\wcone)
    \xrightarrow{\partial}\cdots.
  \end{equation}
  Here we are using the opposite sign convention from the Mayer--Vietoris
  sequence of \cref{lem:M-V on naturals}, as this results in a tidier
  computation.
  We use the natural isomorphism from \cref{lem:coarse homotopy equivalence}
  and the exact sequence from \cref{lem:M-V on naturals} to obtain a
  commutative diagram as follows:
  \begin{equation}\label{cd: expnat to sum via top top}
  \begin{tikzcd}
    HX_\ast(\wcone^\expnat)\ar{r}{({\iota_1}_*,{\iota_2}_*)}\ar[dashed]{dr}{D}
    &
    HX_\ast\bigparen{\wcone^{I_1}}\oplus HX_\ast\bigparen{\wcone^{I_2}} 
    \\ 
    HX_\ast\bigparen{\wcone^{\expodd}}\oplus HX_\ast\bigparen{\wcone^{\expeven}}
    \ar[->>]{u}{(\iotaodd)_\ast+(\iotaeven)_\ast}
    \ar[dashed]{r}{A}
    &
    HX_\ast\bigparen{\wcone^{\expodd}}\oplus HX_\ast\bigparen{\wcone^{\expeven}}.
    \ar{u}[swap]{(\topp_1)_\ast \oplus(\topp_2)_\ast}{\cong}
  \end{tikzcd}
  \end{equation}
  Note that
  \begin{align*}
    \iota_1\circ\iotaodd &= \topp_1 
    &
    \iota_2\circ\iotaodd &= \bott_2
    \\
    \iota_1\circ\iotaeven &= \bott_1
    &
    \iota_2\circ\iotaeven &= \topp_2.
  \end{align*}
  Therefore \cref{lem:shift in homology} shows that the arrow labelled with $A$ is given by
  \[
A=  \left(
    \begin{array}{cc}
      1 & s_{\rm ev} \\
      s_{\rm od} & 1
    \end{array}
  \right) .
  \]

  We now identify the right hand side of the diagram \eqref{cd: expnat to sum via top top}, apply the exact sequence of \cref{lem:M-V on naturals}, and then combine it with the long
  exact sequence \eqref{eq:M-V on I1, I2} to obtain the commutative diagram
  \begin{equation}\label{cd: exactness at first expnat}
  \begin{tikzcd}[row sep=3em]
    HX_{\ast-1}(\wcone)\ar{r}{\partial}
    &[-3em]
    HX_\ast(\wcone^\expnat)\ar{r}{\id+\shift_\ast}\ar{dr}{D}
    &[-.5em]
    HX_\ast(\wcone^\expnat) 
    \\
    &
    HX_\ast\bigparen{\wcone^{\expodd}}\oplus HX_\ast\bigparen{\wcone^{\expeven}}
    \ar[->>]{u}{(\iotaodd)_\ast+(\iotaeven)_\ast}
    \ar{r}{A}   
    &
    HX_\ast\bigparen{\wcone^{\expodd}}\oplus HX_\ast\bigparen{\wcone^{\expeven}}.
    \ar{u}[swap]{(\iotaodd)_\ast+(\iotaeven)_\ast}
    \\
    &
    HX_\ast(\pt)\ar{u}{(\iota_\ast,-\iota_\ast)}
    &
    HX_\ast(\pt)\ar{u}[swap]{(\iota_\ast,-\iota_\ast)}
  \end{tikzcd}
  \end{equation}
  Note that the top line is exact at $HX_\ast(\wcone^\expnat)$ because by
  a diagram chase $\ker(D)=\ker(\id+\shift_\ast)$.  
  Combining \eqref{cd: expnat to sum via top top} with the right part of the
  Mayer--Vietoris exact sequence \eqref{eq:M-V on I1, I2}, we also know that
  \begin{equation}\label{eq: exact right of D}
    HX_\ast(\wcone^\expnat)
    \xrightarrow{D}
    HX_\ast\bigparen{\wcone^{\expodd}}\oplus HX_\ast\bigparen{\wcone^{\expeven}}
    \xrightarrow{\paren{j_1\circ \topp_1}_* -\paren{j_2\circ\topp_2}_*}
    HX_\ast(\wcone)
  \end{equation}
  is exact. 
    Combining the above and the short exact sequence of \cref{lem:M-V on naturals} we obtain the following diagram 
    \[\begin{tikzcd}
  & HX_\ast(\wcone^\expnat) \ar[dashed, shift left=1.5ex]{dr} & [6em] \\
    HX_\ast(\wcone^\expnat)
    \ar{r}{D}
    &
    HX_\ast\bigparen{\wcone^{\expodd}}\oplus HX_\ast\bigparen{\wcone^{\expeven}}\ar[->>]{u}[swap]{(\iotaodd)_\ast+(\iotaeven)_\ast}
    \ar{r}{\paren{j_1\circ \topp_1}_* -\paren{j_2\circ\topp_2}_*}
    &[6em]
    HX_\ast(\wcone)
    \\
    &
     HX_\ast(\pt)\ar{u}[swap]{(\iota_\ast,-\iota_\ast)}
    &
  \end{tikzcd},
  \]
  where the vertical and horizontal sequences are exact.
  The image of $HX_\ast(\pt)$ under the composition of
  $(\iota_\ast,-\iota_\ast)$ and the map on the right hand side of \eqref{eq:
  exact right of D} is $2 \cdot \iota_\ast(HX_\ast(\pt))$. However, we have a
  factorization of $\iota$ as $\{\pt\}\to [1,\infty)\to \wcone$. As
  $HX_*([1,\infty))=0$ by flasqueness, we see that $2\iota_*HX_*(\pt)=0$. It follows that
  $\paren{j_1\circ \topp_1}_* -\paren{j_2\circ\topp_2}_*$ factors 
  through the quotient by $(\iota_\ast,-\iota_\ast)(HX_\ast(\pt))$, as indicated by the dashed arrow in the diagram above. We call the quotient map $\jalt$ and observe that the top row exact sequence of
  \eqref{cd: exactness at first expnat} prolongs to the right of
  $(\id+\shift_\ast)$:
  \begin{equation}\label{eq: diagram prolong exactness to the right}
  \begin{tikzcd}[row sep=3em]
    HX_\ast(\wcone^\expnat)\ar{r}{\id+\shift_\ast}\ar{dr}{D}
    &
    HX_\ast(\wcone^\expnat)\ar{r}{\jalt}
    &[0em]
    HX_\ast(\wcone)\ar{r}{\partial}
    & \cdots
    \\
    &
    HX_\ast\bigparen{\wcone^{\expodd}}\oplus HX_\ast\bigparen{\wcone^{\expeven}}
    \ar[->>]{u}[swap]{(\iotaodd)_\ast+(\iotaeven)_\ast}
    \ar{ur}[swap]{\paren{j_1\circ \topp_1}_* -\paren{j_2\circ\topp_2}_*}
    & &
  \end{tikzcd}
  \end{equation}
  This concludes the proof of \cref{thm:LES}.
\end{proof}

\section{Consequences of \cref{thm:LES}}\label{sec: consequences of LES}
It is well known that both the coarse K-homology and the K\=/theory of Roe
algebras satisfy the three axioms listed in \cref{ssec:coarse homology setup}.
For references, we refer to \cite[Section 8]{BunkeEngel2020HomotopyTheory}, in
particular \cite[Theorems 8.79 and 8.88]{BunkeEngel2020HomotopyTheory}, which
shows that the the K-theory of the Roe algebra is a coarse homology theory in
the sense of \cite{BunkeEngel2020HomotopyTheory}. On the other hand,
\cite[Section 7]{BunkeEngel2020HomotopyTheory}, in particular \cite[Proposition
7.46,Definition 7.66]{BunkeEngel2020HomotopyTheory} implies that all
coarsifications of locally finite homology theories and therefore also coarse
K-homology is a coarse homology theory in the sense of
\cite{BunkeEngel2020HomotopyTheory}. By \cite[Corollary
4.28]{BunkeEngel2020HomotopyTheory} a coarse homology theory in the sense of
\cite{BunkeEngel2020HomotopyTheory} satisfies homotopy invariance and
Mayer--Vietoris (the latter is the long exact sequence of homotopy groups
associated to the cocartesian square of spectra of \cite[Corollary 4.28
4.]{BunkeEngel2020HomotopyTheory}).

We need one further property which is not completely automatic: the
compatibility of the coarse assembly map with the boundary map of the
Mayer--Vietoris sequences. This is commonly used in the literature and proven in
\cite[Section 3]{Siegel} or \cite[Theorem 2.10]{SchickZadeh}.

In the following, which is well known to experts, we verify also the last requirement of \cref{thm:LES}:

\begin{lemma}
  The inclusion $\{1\}\hookrightarrow \expnat$ induces injections
  \[
   KX_\ast(\{1\})\hookrightarrow KX_\ast(\expnat)
   \qquad
   K_\ast(\Croe{\{1\})}\hookrightarrow K_\ast(\Croe{\expnat})
  \]
\end{lemma}
\begin{proof}
  The coarse K-homology of $X$ can be defined as the direct limit of the
  locally finite K-homology groups of the $R$-Rips complexes $P_R(X)$
  as $R$ tends to infinity. Passing to an appropriate cofinal sequence $R_n$,
  the associated Rips complex consists of the simplex spanned by the
    first $n+1$ points together with the disjoint union of the points $2^k$
    with $k>n$. Since we use locally finite  K-homology, $K_\ast(P_{R_n}(X))\cong \prod_{k\geq n} K_\ast(\pt)$ (alternatively, see \cite{WillettYu2020HigherIndex}*{Theorem 6.4.20}).

  For $n<m$, the inclusion $P_{R_{n}}(X)\hookrightarrow P_{R_{m}}(X)$ induces the homomorphism
  \begin{equation*}
    \begin{tikzcd}[column sep = 0pt,row sep = 0pt]
      \varphi_{n,m}\colon &[-.5em] \prod_{k\geq n} K_\ast(\pt) \ar{r} &[2em] \prod_{k\geq m} K_\ast(\pt);
      \\ 
      & (a^{(n)}_k)_{k\geq n}\ar[|->]{r} & (\sum_{k=n}^{m}a^{(n)}_k,\;a^{(n)}_{m+1},\;a^{(n)}_{m+2},\dots).
    \end{tikzcd}  
  \end{equation*}that adds up the components between $n$ and $m$.

  Let $\sigma\colon \bigoplus_{k\in \naturals}K_\ast(\pt)\to
  K_\ast(\pt)$ and $\sigma_n\colon \bigoplus_{k\ge n} K_\ast(\pt)\to
  K_\ast(\pt)$ be the sum homomorphisms. The inclusions (extending by
  zero) then induce isomorphisms
  \begin{equation*}
    \frac{\prod_{k\geq n} K_\ast(\pt)}{\ker(\sigma_n)}
      \xrightarrow{\ \cong\ }
    \frac{\prod_{k\in \naturals} K_\ast(\pt)}{\ker(\sigma)}. 
  \end{equation*}
  These homomorphisms are compatible with the $\varphi_{n,m}$ and hence induce
  a well defined homomorphism
  \begin{equation*}
      KX_\ast(\expnat) 
    \cong \varinjlim_{n\in\mathbb N}\left(\prod_{k\ge n} K_\ast(\pt)\right)
    \longrightarrow  \frac{\prod_{n\in \mathbb N} K_\ast(\pt)}{\ker(\sigma)}.
  \end{equation*}
  This map is surjective, as already the homomorphism for $n=0$ is
  surjective. It is also injective, because every element of $\ker(\sigma)$ is
  mapped to $0$ by some $\varphi_{0m}$. Using the isomorphism $\frac{\bigoplus_{n\in\naturals}
    K_*(\pt)}{\ker(\sigma)} \xrightarrow{\cong} K_*(\pt)$ and the isomorphism
  theorem,  we get the canonical short exact sequence
  \begin{equation}\label{eq:compute_KX}
    0\to K_\ast(\pt)\longrightarrow KX_*(2^\naturals) \longrightarrow \frac{\prod_{n\in\naturals}
      K_*(\pt)}{\bigoplus_{n\in\naturals}K_*(\pt)} \to 0
  \end{equation}
  The homomorphism induced by inclusion $\{1\}\to \expnat$ factorizes---using
  $KX_\ast(\{1\})\cong K_\ast(\pt)$---through the inclusion of $K_\ast(\pt)$
  in \eqref{eq:compute_KX} and is therefore injective.
  
  The proof for the $K$-theory of the Roe algebra is analogous by using that
  \[
  \Croe{\expnat}\cong \varinjlim_{n\in\mathbb N}\left(\prod_{k\geq n}\Croe{\pt} \right)
  \cong \varinjlim_{n\in\mathbb N}\left(\prod_{k\geq n}\mathcal K \right)
  \]
  and continuity of K\=/theory. Here the description as limit is obtained by observing that the set of locally compact operators of propagation at most $R_n$ is identified with the $C^*$-algebra $\prod_{k\geq n}\Croe{\pt}$, where the first instance of $\Croe{\pt}$ accounts for all the operators supported on $\{2^k\mid k\leq n\}$.
  For $n<m$, the homomorphisms $\varphi_{n,m}$ induced in K-theory by the inclusions $\{\text{prop.}\leq R_n\}\hookrightarrow \{\text{prop.}\leq R_m\}$ are the sum, just as in the K-homology case. The same computation then applies.
  
  Alternatively, the statement in K-theory can be deduced from the statement on coarse K-homology, because $\expnat$ satisfies the coarse Baum--Connes conjecture.
\end{proof}

We can hence apply \cref{thm:LES} to obtain long exact sequences in this setting. This is done in the next two subsections.

\subsection{Vanishing revisited}
As a first application, we note that \cref{thm:LES} can be used to give yet another proof of the fact that the
(integral) Dru\k{t}u--Nowak projection vanishes in $K$-theory.

Recall that we construct the Roe algebra of $\wcone^\naturals$ using the geometric module $L^2(M)\otimes\ell^2(\naturals)$, which contains one orthogonal copy of $L^2(M)$ for each $n\in \naturals$.
Let $p\in \mathcal B(L^2M)$ be a projection such that
$p\otimes\id_{\ell^2\naturals}\in\Croe{\wcone^\naturals}$.\footnote{ We already
observed that every $p\in\Cfp{\Gamma\curvearrowright M}\cap \mathcal K(L^2M)$
satisfies this condition, but a priori this might be true for other projections
as well. }
It is convenient to write $p\otimes\id_{\ell^2\naturals}$ as the
sequence $(p)_{i\in \naturals} = (p,p,\ldots)$.
With this notation, when the integral Dru\k{t}u--Nowak projection $\mathfrak G^\naturals$ exists, it is given by $(m,m,\ldots)$, where $m\in\mathcal B(L^2M)$ is the projection onto the constant functions.

Given $p$ as above and a sequence of integers $(c_i)_{i\in \naturals}$, we wish to consider the K\=/theory class
$[(c_0p,c_1p,\ldots)]\in K_0(\Croe{\wcone^\naturals})$.
To make sense of this, if the $c_i$'s are all positive one may tensor $L^2(M)\otimes\ell^2(\naturals)$ with another infinite dimensional separable Hilbert space $\mathcal H$ and define $c_i(p\otimes\delta_i)$ as $p\otimes\delta_i\otimes q_i$ where $q_i\in \mathcal B(\mathcal H)$ is a projection of rank $c_i$ (i.e.\ we pass to a ``very ample'' geometric module). The K\=/theory class 
\[
  [(c_0p,c_1p,\ldots)] \coloneqq [(p\otimes\delta_0\otimes q_0\;,\ p\otimes\delta_1\otimes q_1\;,\; \ldots\; )]
  \in K_0(\Croe{\wcone^\naturals})
\]
does not depend on the choices made: had we used $q_i'\in\mathcal B(\mathcal H)$, we could pick unitaries $U_i\in \U(\mathcal H)$ conjugating $q_i'$ to $q_i$, and assemble them into a unitary operator $U$ on $L^2(M)\otimes\ell^2(\naturals)\otimes\mathcal H$ that conjugates $(p\otimes\delta_i\otimes q_i')_{i\in \mathbb N}$ to $(p\otimes\delta_i\otimes q_i)_{i\in \mathbb N}$. Since $U$ has propagation zero, this shows that the two projections define the same element in K-theory.
Note that $[(1\,p,1\,p,\ldots)]$ is naturally identified with $[(p,p,\ldots)]$ by embedding $L^2(M)\otimes\ell^2(\naturals)$ into $L^2(M)\otimes\ell^2(\naturals)\otimes\mathcal H$.

A K\=/theory class $[(c_0p,c_1p,\cdots)]$ with non-positive coefficients is obviously defined taking formal differences.
Identifying $K_0(\Croe{\wcone^\naturals})\cong K_0(\Croe{\wcone})$ via the embedding $\wcone^\naturals\subseteq\wcone$ as at the beginning of \cref{sec: vanishing proof}, we can then make sense of the class $[(c_ip)_{i\in\naturals}]\in K_0(\Croe{\wcone})$.

\begin{remark}
  Using this notation, the class $[\Phi(\mathfrak G^\naturals)]$ appearing in \cref{ssec: swindle} is nothing but $[((n+1)m)_{n\in\naturals}]$.
\end{remark}

We can do the same on the exponentially-spaced warped cone $\wcone^\expnat$: given $p$ as above and an integer sequence $(c_{2^j})_{2^j\in\expnat}$ we obtain a class $[(c_{2^j}p)_{2^j\in\expnat}]\in K_0(\Croe{\wcone^\expnat})$.
Quite plainly, we have:
\begin{equation}\label{eq: shift of K0}
  \shift_\ast([(c_{2^0}p,c_{2^1}p,\cdots)]) = [(0,c_{2^0}p,c_{2^1}p,\cdots)]
   = [(c_{2^{j-1}}p)_{2^j\in 2^{\naturals+1}}].
\end{equation}

\begin{remark}

  One should be careful when using sequences as above to denote K-theory classes: it is generally not clear when a K-theory class can be represented by such a sequence, and when it does, this representation is generally non-unique.

  On the other hand, since we gave an explicit construction for the class
  $[(c_ip)_{i\in\naturals}]$ in terms of (differences of) projections, here it
  is easy to describe its image under homomorphisms induced by functions at
  the level of spaces.

  For instance, \eqref{eq: shift of K0} is justified by observing that the isometry
  \[
    \begin{tikzcd}[column sep = 0pt,row sep = 0pt]
      V\colon &[-.5em] L^2(M)\otimes\ell^2(\expnat)\otimes\mathcal H
       \ar{r} &[2em] L^2(M)\otimes\ell^2(\expnat)\otimes\mathcal H;
      \\ 
      & \xi\otimes \delta_{2^i}\otimes\eta \ar[|->]{r} & \xi\otimes \delta_{2^{i+1}}\otimes\eta.
    \end{tikzcd}  
  \]
 covers the shift map, and hence $\shift_\ast$ is the homomorphism induced in K-theory by the conjugation $t\mapsto VtV^*$. The latter acts as prescribed in \eqref{eq: shift of K0} on the defining projections $(p\otimes\delta_i\otimes q_i)_{i\in \expnat}$.

  Similar arguments will be used again in the sequel. In \cref{lem: concentrate D-N} we also need to construct isometries covering non-injective maps of spaces. This is simply done by arbitrarily embedding $\mathcal H \otimes\ell^2(\naturals)$ into $\mathcal H$ to create enough space for it.
\end{remark}

The following is now easy to prove:

\begin{lemma}
  Any class $x = [(c_{2^0}p,c_{2^1}p,\ldots)]\in K_0(\Croe{\wcone^\expnat})$ as above is in the kernel of
  $\jalt$.
\end{lemma}
\begin{proof}
  Consider the class 
  \[
  y\coloneqq\sum_{n\in\mathbb N}(-1)^n\shift_\ast^{n}(x) = [(c_{2^0}p, (c_{2^1}-c_{2^0})p, (c_{2^2}-c_{2^1}+c_{2^0})p,\ldots )]\in K_0(\wcone^\expnat)
  \]
  (the infinite series is simply intended as a formal sum, and it is seen to make sense because the RHS is well-defined).
  Then $x= (\id+\shift_\ast)(y)$, and therefore
  $\jalt(x)=0$ by exactness.
\end{proof}

From the diagram \eqref{eq: diagram prolong exactness to the right}, we read that $\jalt\colon K_0(\Croe{\wcone^\expnat})\to K_0(\Croe{\wcone})$ maps
\[
[(c_{2^{j}}p)_{2^j\in \expnat}] \longmapsto [((-1)^{j+1}c_{2^{j}}p)_{2^j\in \expnat}],
\]
where it is understood that in the right hand side every element in $\naturals\smallsetminus\expnat$ has coefficient zero.
Here we are being ever so slightly imprecise in that \eqref{eq: diagram prolong exactness to the right} does not directly specify what happens to the coefficient $c_{2^0}$. However, the fact that the quotient map $\jalt$ is well-defined means precisely that this choice does not matter.

Suppose now that $\Gamma\curvearrowright M$ is strongly ergodic, and let
$\mathfrak G^{\mathbb N} = (m)_{i\in \naturals}$ be the integral Dru\k{t}u--Nowak projection.
We note that this is equivalent to a class concentrated on the exponential level sets:

\begin{lemma}\label{lem: concentrate D-N}
  \(
[(m)_{i\in \naturals}] = [(2^jm)_{2^j\in\expnat}].
\)
\end{lemma}
\begin{proof}
  We split up the cone into exponential intervals and consider their coarse disjoint union
  \[
  Y \coloneqq \bigsqcup_{n\in\naturals}\wcone^{[2^n,2^{n+1})}.
  \]
  We define the class
  \[
  [\mathfrak G^\naturals_{\rm split}]\coloneqq [(m),(m,m),(m,m,m,m),\ldots]\in K_0(\Croe{Y})
  \]
  the obvious way, and note that $[\mathfrak G^\naturals]$ is the image of $[\mathfrak G^\naturals_{\rm split}]$ under the map 
  \[f_0\colon Y\to\wcone\]
  that glues the intervals back up (this map is proper and controlled).

  On the other hand, also the mapping $f_1\colon Y\to\wcone$ defined by compressing each interval to their bottom extremity is proper and controlled. Since $f_0$ and $f_1$ are clearly coarse homotopic (cf.\ \cref{lem:coarse homotopy equivalence}), they must induce the same map in K-theory. Then we are done, because $(f_1)_\ast([\mathfrak G^\naturals_{\rm split}])=[(2^jm)_{2^j\in\expnat}]$.
\end{proof}

The following is now clear.

\begin{corollary}\label{cor: vanishing II}
  Let $\Gamma\curvearrowright M$ be strongly ergodic.
  Then $[\mathfrak G^{\mathbb N}]=0$ in $K_0(\Croe{\wcone})$.
\end{corollary}
\begin{proof}
  $[\mathfrak G^{\mathbb N}]= [(2^jm)_{2^j\in\expnat}]$ is the image under ${\jalt}$ of
  \([((-1)^{j+1}2^{j}m)_{2^j\in \expnat}]\).
\end{proof}

\begin{remark}
  Of course, in \cref{cor: vanishing II} there is nothing special about $[\mathfrak G^{\mathbb N}]$, the same proof applies to any class of the form $[(c_ip)_{i\in \naturals}]$.
  It is most likely possible to use \cref{thm:LES} to directly prove that all
  generalized Dru\k{t}u--Nowak projections vanish in K\=/theory. However, this
  would require various technical details to deal with projections of the form
  $m\otimes q$ where $q$ is something more complicated than
  $\id_{\ell^2\naturals}$ (e.g.\ it does not have finite propagation). In this
  case \cref{thm:vanishing} seems to be the most straightforward solution.
\end{remark}

\subsection{Failure of coarse Baum--Connes}
As a second and main application, we provide the announced counterexamples to coarse Baum--Connes in K$_1$.

\begin{proof}[Proof of \cref{thm:cBC_counter}]
By assumption, $\Gamma$ is a group with property A and $\Gamma\curvearrowright M$ a free strongly ergodic action by Lipschitz homeomorphisms.
We then know that the averaging
projection $m\in\mathcal B(L^2M)$ lies in $\Cfp{\Gamma\curvearrowright M}\cap \mathcal
K(L^2(M))$, and hence both $(m,0,m,0,\dots)$ and $(0,m,0,m,\dots)$ belong to $\Croe{\wcone^{2^\naturals}}$.
The alternating Dru\k{t}u--Nowak class
\[
  [\mathfrak G^\expnat_{\rm alt}] \coloneqq [(m,0,m,0,\dots)] - [(0,m,0,m,\dots)]
  = [(m,-m,m,-m,\ldots)]
\]
is then an element in $K_0(\Croe{\wcone^\expnat})$.

The same methods used to prove \cref{thm: old-counterexample}, also prove that
$[\mathfrak G^\expnat_{\rm alt}]$ does not belong to the image of $\mu_c\colon KX_0(\wcone^\expnat)\to K_0(\Croe{\wcone^\expnat})$.
In fact, exactly as in \cref{thm: old-counterexample}, we may consider the two traces
\[
\tau_{\rm d}, \tau_{\rm u}\colon K_0(\Croe{\wcone^\expnat})
\to \frac{\prod_{n\in\mathbb N}\mathbb R}{\bigoplus_{n\in\mathbb N}\mathbb R},
\]
which we know coincide on the image of $\mu_c$ (the existence of $\tau_{\rm u}$ is the only place where we need the property A, Lipschitz and freeness assumptions).
We then observe that 
\[
  \tau_{\rm d}([\mathfrak G^\expnat_{\rm alt}])= [(-1,\;1,-1,\;1,\ldots)]\neq 0,
\]
while $\tau_{\rm u}([\mathfrak G^\expnat_{\rm alt}])=0$ because $[\mathfrak
G^\expnat_{\rm alt}]$ is the difference of two classes represented by ghost projections.

Having established this, we wish to use $[\mathfrak G^\expnat_{\rm alt}]$ to construct an element in 
\begin{equation*}
  \ker\left(\id+\shift_*\colon 
    K_0(\Croe{\wcone^{2^\naturals}})\to 
    K(\Croe{\wcone^{2^\naturals}})\right).
\end{equation*}
Since $(\id+\shift_\ast)([\mathfrak G^\expnat_{\rm alt}])=[(m,0,0,\ldots)]$, this is easy to do:
arbitrarily pick some $x\in M$, and for every $n\in\mathbb N$ pick a rank-1 orthogonal projection $p_n\in \mathcal B(L^2M)$ such that the image of $p_n\otimes\delta_{2^n}$ is supported on the ball of radius $1$ centered at $(x,2^n)$.
This gives two further projections $(p_0,0,,p_2,0,\dots )$ and $(0,p_1,0,p_3,\dots)$ in $\Croe{\wcone^{2^\naturals}}$
and a class
\[
[p_{\rm alt}]\coloneqq [(p_0,0,p_2,0,\dots )] - [(0,p_1,0,p_3,\dots)]
\in K_0(\Croe{\wcone^{2^\naturals}}).
\]
Then $(\id+\shift_\ast)([p_{\rm alt}])=[(p_0,0,0,\ldots)]$, because for every $n\in\mathbb N_{\geq 1}$ the projections $p_{n-1}\otimes\delta_{2^{n}}$ and $p_n\otimes\delta_{2^n}$ are unitarily equivalent via a
unitary of propagation at most $3$.
Now, the image of $[\mathfrak G^\expnat_{\rm alt}]-[p_{\rm alt}]$ under $\id+\shift_*$ is represented by
\begin{equation*}
  [(m-p_0,0,0,\dots)] =0
\end{equation*}
which vanishes, because on the single slice $\wcone^{\{1\}}$ the projections $m$
and $p_0$ are unitarily equivalent via a unitary of finite propagation.

At this point we are done. The class $[p_{\rm alt}]$ is in the image of
$\mu_c$, and hence $[\mathfrak G^\expnat_{\rm alt}]-[p_{\rm alt}]$ is not.
The compatibility of  of the coarse assembly with the boundary maps of
  the Mayer--Vietoris sequences shows that we have a commutative diagram
\[
\begin{tikzcd}
  KX_1(\wcone)\ar[->>]{r}\ar{d}{\mu_c}
  &
  \ker\left(\id+\shift_*\colon 
    K_0(KX_0{\wcone^{2^\naturals}})\to 
    K(KX_0{\wcone^{2^\naturals}})\right)
  \ar{d}{\mu_c}
  \\
  K_1(\Croe{\wcone})\ar[->>]{r}
  &
  \ker\left(\id+\shift_*\colon 
    K_0(\Croe{\wcone^{2^\naturals}})\to 
    K(\Croe{\wcone^{2^\naturals}})\right).
\end{tikzcd}
\]
An immediate diagram chase then proves that any lift of the class $[\mathfrak G^\expnat_{\rm alt}]-[p_{\rm alt}]$ to $K_1(\Croe{\wcone})$ does not belong to the image of 
\[\mu_c \colon KX_1(\wcone) \longrightarrow   K_1(\Croe{\wcone}). \qedhere\]
\end{proof}

\begin{remark}
  Note that in the construction of $[p_{\rm alt}]$ we could have chosen $p_n$ to be supported on the $1$-ball around an arbitrary point $(x_n,2^n)$, and the resulting K\=/theory class would have stayed the same.
  In fact, this class is just the image of a K\=/theory class in $K_0(\Croe{2^n})$ under the embedding $2^\naturals\into
  \wcone^{2^\naturals}$ sending $2^n$ to $(x_n,2^n)$. Since $M$ is compact and connected, it is easily seen that any two such embedding are coarsely homotopic.
\end{remark}

\bibliography{biblio.bib}

\end{document}

%% file: GenCommands.tex
\usepackage{amsmath, amsthm, amssymb}
\usepackage[T1]{fontenc}
\usepackage[utf8]{inputenc}     
\usepackage[dvipsnames]{xcolor}
\usepackage{xstring}            
\usepackage{enumerate}
\usepackage[
  colorlinks=true, linkcolor=blue, linkbordercolor=blue, citecolor=red, citebordercolor=red, urlcolor=blue, linktocpage=true
]{hyperref}
\usepackage[capitalise, nameinlink, noabbrev, nosort]{cleveref} 
\usepackage[lite]{amsrefs}  

\usepackage{mathtools}          
\usepackage{bm}           
\usepackage[shortcuts]{extdash}       
\usepackage{subcaption}         
\usepackage{accents}
\usepackage{xspace}

\usepackage{xparse} 

\usepackage{tikz-cd}
\usetikzlibrary{decorations.markings}
\tikzset{double line with arrow/.style args={#1,#2}{decorate,decoration={markings,%
mark=at position 0 with {\coordinate (ta-base-1) at (0,-2pt);
\coordinate (ta-base-2) at (0,2pt);},
mark=at position 1 with {\draw[#1] (ta-base-1) -- (0,-2pt);
\draw[#2] (ta-base-2) -- (0,2pt);
}}}}

\DeclarePairedDelimiter\norm{\lVert}{\rVert}        
\DeclarePairedDelimiter\angles{\langle}{\rangle}    

\DeclarePairedDelimiter\paren{(}{)}            
\DeclarePairedDelimiter\braces{\{}{\}}            

    \newcommand{\bigparen}[1]{\paren[\big]{#1}}

\DeclareMathOperator{\id}{id}                
\DeclareMathOperator{\Ad}{Ad}

\DeclareMathOperator{\supp}{supp}            

\DeclareMathOperator{\im}{Im}                

\DeclareMathOperator{\SU}{SU}            
\DeclareMathOperator{\U}{U}

\DeclareMathOperator{\dd}{d}

\theoremstyle{plain}
\newtheorem{thm}{Theorem}[section]        \newtheorem{theorem}[thm]{Theorem}
        
            \newtheorem{lemma}[thm]{Lemma}
        \newtheorem{corollary}[thm]{Corollary}

\newtheorem*{thm*}{Theorem}            \newtheorem*{theorem*}{Theorem}
\newtheorem*{prop*}{Proposition}        \newtheorem*{proposition*}{Proposition}
\newtheorem*{lem*}{Lemma}            \newtheorem*{lemma*}{Lemma}
\newtheorem*{cor*}{Corollary}            \newtheorem*{corollary*}{Corollary}
\newtheorem*{qu*}{Question}            \newtheorem*{question*}{Question}
\newtheorem*{conj*}{Conjecture}            \newtheorem*{conjecture*}{Question}
\newtheorem*{prob*}{Problem}        \newtheorem*{problem*}{Problem}
\newtheorem*{fact*}{Fact}
\newtheorem*{claim*}{Claim}
\newtheorem*{case*}{Case}

\numberwithin{equation}{section}

\newtheorem{alphthm}{Theorem}            

\newtheorem{alphcor}[alphthm]{Corollary}           

\newtheorem{alphconj}[alphthm]{Conjecture}

\theoremstyle{definition}
        \newtheorem{definition}[thm]{Definition}

\newtheorem*{de*}{Definition}            \newtheorem{definition*}{Definition}
\newtheorem*{notation*}{Notation}
\newtheorem*{conv*}{Convention}            \newtheorem*{convention*}{Convention}

\theoremstyle{remark}
            \newtheorem{remark}[thm]{Remark}
            \newtheorem{example}[thm]{Example}

\newtheorem*{rmk*}{Remark}        \newtheorem*{remark*}{Remark}

\DeclareMathAlphabet{\mathbit}{OT1}{cmr}{bx}{it}

\crefname{thm}{Theorem}{Theorems}              \crefname{theorem}{Theorem}{Theorems}
\crefname{prop}{Proposition}{Propositions}     \crefname{proposition}{Proposition}{Propositions}
\crefname{lem}{Lemma}{Lemmas}                  \crefname{lemma}{Lemma}{Lemmas}
\crefname{rmk}{Remark}{Remarks}                \crefname{remark}{Remark}{Remarks}
\crefname{cor}{Corollary}{Corollaries}         \crefname{corollary}{Corollary}{Corollaries}
\crefname{qu}{Question}{Questions}             \crefname{question}{Question}{Questions}
\crefname{conj}{Conjecture}{Conjectures}       \crefname{conjecture}{Conjecture}{Conjectures}
\crefname{prob}{Problem}{Problems}             \crefname{problem}{Problem}{Problems}
\crefname{fact}{Fact}{Facts}
\crefname{claim}{Claim}{Claims}
\crefname{case}{Case}{Cases}
\crefname{alphthm}{Theorem}{Theorems}          \crefname{alphcor}{Corollary}{Corollaries}
\crefname{alphprop}{Proposition}{Propositions}

\newcommand{\Cat}[1]{\ifmmode \text{\normalfont \textbf{#1}} \else {\normalfont \textbf{#1}}\fi}
\newcommand{\mhyphen}{\textnormal{-}}
\newcommand{\variable}{\,\mhyphen\,}

\newcommand\MakeComments[3]{
  \newcounter{#2comment}\addtocounter{#2comment}{1}%
  \newcommand{#1}[1]{%
     \textbf{\color{#3}(\StrChar{#2}{1}\arabic{#2comment})}%
     \marginpar{\tiny\raggedright\textbf{\color{#3}(\StrChar{#2}{1}\arabic{#2comment}) #2:} ##1}%
    \addtocounter{#2comment}{1}%
  }%
}

\definecolor{newpurple}{rgb}{0.8, 0, 0.9}